\documentclass{article}
\usepackage{amsmath,amssymb,bbm,theorem}

\newcommand{\assign}{:=}
\newcommand{\nocomma}{}
\newcommand{\noplus}{}
\newcommand{\tmem}[1]{{\em #1\/}}
\newcommand{\tmop}[1]{\ensuremath{\operatorname{#1}}}
\newenvironment{proof}{\noindent\textbf{Proof\ }}{\hspace*{\fill}$\Box$\medskip}
\newtheorem{corollary}{Corollary}
\newtheorem{lemma}{Lemma}
{\theorembodyfont{\rmfamily}\newtheorem{remark}{Remark}}
\newtheorem{theorem}{Theorem}
\newcommand{\XXint}[3]{{\setbox}0=\text{\ensuremath{#1 #2 #3 \int}}
{\vcenter{\text{\ensuremath{#2 #3}}}}{\kern}-.5{\tmwd}0}

\newcommand{\opn}[2]{\newcommand{\1}{\}} {\opn}{\Rm{Rm}} {\opn}{\Ric{Ric}}
{\opn}{\Rc{Rc}} {\opn}{\Scal{Sc}} {\opn}{\Tr{Tr}} {\opn}{\Trac{Tr}}
{\opn}detdet {\opn}{\diam{diam}} {\opn}{\dist{dist}} {\opn}{\Im}Im
{\opn}{\div}div {\opn}{\Ker{Ker}} {\opn}expexp {\opn}{\Vol{Vol}}
{\opn}{\exph{exph}} {\opn}{\Herm{Herm}} {\opn}{\End{End}} {\opn}{\Hess{Hess}}
{\opn}{\Vol{Vol}}}

\newcommand{\R}{\mathbb{R}}

\newcommand{\contract}{{\kern}-1.5pt{\vrule} width6.0pt height0.4pt depth0pt
{\vrule} width0.4pt height4.0pt depth0pt}
\newcommand{\retract}{{\kern}-1.5pt{\vrule} width0.4pt height4.0pt depth0pt
{\vrule} width6.0pt height0.4pt depth0pt}
\newcommand{\Openbox}{{\leavevmode} {\hfil}{\vrule} width{\boxrulethickness}
{\vbox} to{\Openboxwidth{{\advance}{\Openboxwidth} -2{\boxrulethickness}
{\hrule} height {\boxrulethickness} width{\Openboxwidth}{\vfil} {\hrule}
height{\boxrulethickness}}}{\vrule} width{\boxrulethickness}{\hfil} }

\begin{document}

\title{On complex deformations of K\"ahler-Ricci solitons}\author{\\
{{NEFTON PALI}}}\date{}\maketitle

\begin{abstract}
  We obtain a formal obstruction, i.e. a necessary condition for the existence of polarised
  complex deformations of K\"ahler-Ricci solitons. This obstruction is expressed in terms of the
  harmonic part of the variation of the complex structure.
\end{abstract}

\section{The obstruction result}

Despite the remarkable work of Podesta-Spiro, \cite{Po-Sp}, not much is known on the existence of complex
deformations of K\"ahler-Ricci solitons. In this paper, we provide an effective
result on this topic. Namely, given any polarised family of complex
deformations over a K\"ahler-Ricci soliton (polarised by the symplectic form
of the initial K\"ahler-Ricci soliton), we can effectively establish a necessary
condition for this family to exist.

Let $\left( X, J, g, \omega \right)$ be a Fano manifold with $\omega =
\tmop{Ric}_{_J} \left( \Omega \right)$, where $\Omega > 0$ is the unique volume form such that  $\int_X \Omega = 1$. 
(We denote by $\tmop{Ric}_{_J} \left( \Omega \right)$ the Chern-Ricci form associated to the volume form $\Omega$).
We introduce the $\Omega$-divergence operator acting on vector fields
$\xi$ as
\begin{eqnarray*}
  \tmop{div}^{\Omega} \xi & \assign & \frac{d (\xi \neg \Omega)}{\Omega} .
\end{eqnarray*}
It is well known (see \cite{Fut}), that the Lie algebra of
$J$-holomorphic vector fields $H^0(X,T^{1,0}_{X,J})$ identifies with the space of complex valued functions
$$
\Lambda^{\Omega}_{g, J}:=\overline{-\tmop{div}^{\Omega}H^0(X,T^{1,0}_{X,J})}\subset C_{\Omega}^{\infty} (X, \mathbbm{C})_0 ,
$$
where
$C_{\Omega}^{\infty} (X, \mathbbm{C})_0$ is the space of smooth complex valued
functions with vanishing integral with respect to $\Omega$.
We denote by
$\mathcal{H}_{g, \Omega}^{0, 1} \left( T_{X, J} \right)$ the space of
$T_X$-valued $\left( 0, 1 \right)$-forms which are harmonic with respect to
the Hodge-Witten Laplacian determined by the volume form $\Omega$. 

Assume now $\left( X, J, g, \omega \right)$ is a compact K\"ahler-Ricci soliton
and consider the functions $f \assign \log \frac{d V_g}{\Omega}$, $F \assign f - \int_X f
\Omega$. The solution of the variational stability problem in \cite{Pal2} shows that
the vanishing harmonic cone
\begin{eqnarray*}
  \mathcal{H}_{g, \Omega}^{0, 1} \left( T_{X, J} \right)_0 & \assign & \left\{
  A \in \mathcal{H}_{g, \Omega}^{0, 1} \left( T_{X, J} \right) \mid \int_X
  \left| A \right|^2_g F\, \Omega = 0 \right\},
\end{eqnarray*}
is relevant for the deformation theory of compact K\"ahler-Ricci solitons. In
the Dancer-Wang K\"ahler-Ricci soliton case $\mathcal{H}_{g, \Omega}^{0, 1}
\left( T_{X, J} \right)_0 \neq \{0\}$, thanks to a result in \cite{Ha-Mu}.

For any $A \in \mathcal{H}_{g, \Omega}^{0, 1} \left( T_{X, J} \right)$ we
define the $\mathbbm{R}$-linear functional
\begin{eqnarray*}
  \Phi_A : \Lambda^{\Omega}_{g, J} & \longrightarrow & \mathbbm{R},\\
  &  & \\
  \Phi_A  \left( u \right) & \assign & \int_X \left[ 2 \tmop{Re} u
  \left\langle \nabla^2_g f, A^2 \right\rangle_g - \left\langle J \nabla_g f
  \neg \nabla_g A, i \,\overline{u} \times_{_J} A_{_{_{_{}}}}
  \right\rangle_{g_{_{}}} \right] \Omega .
\end{eqnarray*}
With these notations we can state our obstruction result.

\begin{theorem}
  \label{main-teo}Let $\left( X, J, g, \omega \right)$ be a compact
  K\"ahler-Ricci soliton, let $\left( J_t, \omega \right)_{t \in \left( -
  \varepsilon, \varepsilon \right)}$ be a smooth family of K\"ahler-Ricci
  solitons with $J_0 = J$ and let $A \in \mathcal{H}_{g, \Omega}^{0, 1} \left(
  T_{X, J} \right)$ be the harmonic part of the variation $\dot{J}_0$. Then $A
  \in \mathcal{H}_{g, \Omega}^{0, 1} \left( T_{X, J} \right)_0$ and $\Phi_A =
  0$.
\end{theorem}

The fact that $A \in \mathcal{H}_{g, \Omega}^{0, 1} \left( T_{X, J} \right)_0$
is a statement in our previous work \cite{Pal2}. We will show also that for any $A \in
\mathcal{H}_{g, \Omega}^{0, 1} \left( T_{X, J} \right)$ holds the identity
\begin{eqnarray*}
  \int_X \left| A \right|_g^2 F \Omega & = & - \int_X \left[ 2 \left\langle
  \nabla^2_g f, A^2 \right\rangle_g - \left\langle J \nabla_g f \neg \nabla_g
  A, J A_{_{_{_{}}}} \right\rangle_{g_{_{}}} \right] \Omega,
\end{eqnarray*}
whose right-hand side shows some similarity with the integral $\Phi_A \left( u
\right)$.

\section{Properties of the first variation of Perelman's $H$ map}

We need to remind a few basic facts proved in \cite{Pal2}. We first remind some of the notations in \cite{Pal2}. Let $\Omega > 0$ be a
smooth volume form over an oriented compact and connected Riemannian manifold
$(X, g)$. We equip the set of smooth Riemannian metrics $\mathcal{M}$ over $X$ with the
scalar product
\begin{equation}
  \label{Glb-Rm-m}  (u, v) \longmapsto \int_X \left\langle \hspace{0.25em} u,
  v \right\rangle_g \Omega,
\end{equation}
for all $u, v \in L^2 (X, S^2_{\mathbbm{R}} T_X^{\ast})$. Let $P_g^{\ast}$ be
the formal adjoint of some operator $P$ with respect to the metric $g$. We
observe that the operator $P^{\ast_{\Omega}}_g \assign e^f P^{\ast}_g  \left(
e^{- f} \bullet \right)$, with $f \assign \log \frac{d V_g}{\Omega}$ , is the
formal adjoint of $P$ with respect to the scalar product (\ref{Glb-Rm-m}). We
define the real weighted Laplacian operator $\Delta^{\Omega}_g \assign
\nabla_g^{\ast_{\Omega}} \nabla_g$.
We notice in particular the identity $\tmop{div}^{\Omega} \nabla_g u = -
\Delta^{\Omega}_g u$, for all functions $u$. 

Over a Fano manifold $\left(
X, J, g, \omega \right)$, 
with $\omega =
\tmop{Ric}_{_J} \left( \Omega \right)$, $\int_X \Omega = 1$.
we define the linear operator $B^{\Omega}_{g, J}$
acting on smooth complex valued functions $u$ as $B^{\Omega}_{g, J} u \assign
\tmop{div}^{\Omega} (J \nabla_g u)$. This is a first order differential
operator. Indeed
\begin{eqnarray*}
  B^{\Omega}_{g, J} u & = & \tmop{Tr}_{\mathbbm{R}} \left( J \nabla^2_g u
  \right) - d f \cdot J \nabla_g u\\
  &  & \\
  & = & g (\nabla_g u, J \nabla_g f),
\end{eqnarray*}
since $J$ is $g$-anti-symmetric. We define the weighted complex Laplacian
operator $\Delta^{\Omega}_{g, J} \assign \Delta^{\Omega}_g - i B^{\Omega}_{J,
g}$, acting on smooth complex valued functions. We remind the identity $\Lambda^{\Omega}_{g, J} = \tmop{Ker} (\Delta^{\Omega}_{g, J} -
2\mathbbm{I})$, (see \cite{Fut}).

We remind now that the
$\Omega$-Bakry-Emery-Ricci tensor of the metric $g$ is defined by the formula
$$
\tmop{Ric}_g (\Omega) \assign \tmop{Ric} (g) \hspace{0.75em} +
\hspace{0.75em} \nabla_g d \log \frac{dV_g}{\Omega}.
$$
A Riemannian metric $g$
is called a $\Omega$-shrinking Ricci soliton if $g = \tmop{Ric}_g (\Omega)$.
We define the following fundamental objects
\begin{eqnarray*}
  h & \equiv & h_{g, \Omega} : = \tmop{Ric}_g (\Omega) - g \nocomma,\\
  &  & \\
  2 H & \equiv & 2 H_{g, \Omega} \assign - \Delta^{\Omega}_g f \noplus \noplus
  + \tmop{Tr}_g h \noplus + 2 f,
\end{eqnarray*}
with $f \assign \log \frac{d V_g}{\Omega}$. We define also the normalised
function $\underline{H} \assign H - \int_X H \Omega$. We denote by
$\mathcal{V}_1$ the space of smooth positive volume forms with unitary
integral over $X$. For any $V \in T_{\mathcal{V}_1}$, we define
$V_{\Omega}^{\ast} \assign V / \Omega$. 

We notice now that over a polarised Fano manifold $(X,\omega)$, $\omega\in
2\pi c_1(X)$, the space of
$\omega$-compatible complex structures $\mathcal{J}_{\omega}$ embeds naturally
inside $\mathcal{M} \times \mathcal{V}_1$ via the Chern-Ricci form. The image
of this embedding is
\begin{eqnarray*}
  \mathcal{S}_{\omega} & \assign & \left\{ (g, \Omega) \in
  \mathcal{M}_{\omega} \times \mathcal{V}_1 \mid \omega = \tmop{Ric}_J
  (\Omega), J = - \omega^{- 1} g \right\},
\end{eqnarray*}
with $\mathcal{M}_{\omega} \assign - \omega \cdot \mathcal{J}_{\omega} \subset
\mathcal{M}$. The fact that the space $\mathcal{J}_{\omega}$ may be singular in
general implies that also the space $\mathcal{S}_{\omega}$ may be singular. We
denote by $\tmop{TC}_{\mathcal{S}_{\omega}, (g, \Omega)}$ the tangent cone of
$\mathcal{S}_{\omega}$ at an arbitrary point $(g, \Omega) \in
\mathcal{S}_{\omega}$. This is by definition the union of all tangent vectors
of $\mathcal{S}_{\omega}$ at the point $\left( g, \Omega \right)$. We notice
that, (see for example \cite{Pal1}), the tangent cone
$\tmop{TC}_{\mathcal{M}_{\omega}, g}$ of $\mathcal{M}_{\omega}$ at an
arbitrary point $g \in \mathcal{M}_{\omega}$ satisfies the inclusion
\begin{equation}
  \tmop{TC}_{\mathcal{M}_{\omega}, g} \subseteq \mathbbm{D}^J_{g, \left[ 0
  \right]}, \label{TConeM}
\end{equation}
with
\begin{eqnarray*}
  \mathbbm{D}^J_{g, [0]} & \assign & \{ v \in C^{\infty} \left( X,
  S_{\mathbbm{R}}^2 T^{\ast}_X \right)_{_{_{_{_{_{}}}}}} \mid \hspace{0.25em}
  v = - J^{\ast} v J, \overline{\partial}_{T_{X, J}}^{} v_g^{\ast} = 0 \},
\end{eqnarray*}
with $v^{\ast}_g \assign g^{- 1} v$. It has been showed in \cite{Pal2} that
for any $(g, \Omega) \in \mathcal{S}_{\omega}$ holds the inclusion
\begin{equation}
  \tmop{TC}_{\mathcal{S}_{\omega}, (g, \Omega)} \subseteq \mathbbm{T}^J_{g,
  \Omega}, \label{TConeS}
\end{equation}
with
\begin{eqnarray*}
  \mathbbm{T}^J_{g, \Omega} & : = & \left\{ (v, V) \in \mathbbm{D}^J_{g,
  \left[ 0 \right]} \times T_{\mathcal{V}_1} \mid 2 d d_J^c V_{\Omega}^{\ast}
  = - d \left( \nabla_g^{\ast_{\Omega}} v_g^{\ast} \neg \omega
  \right)_{_{_{_{_{_{}}}}}}  \right\} .
\end{eqnarray*}
(We will use the definition $2 d_J^c \assign i ( \overline{\partial}_J -
\partial_J)$ in this paper). We remind (see \cite{Pal2}) that a point $\left( g,
\Omega \right) \in \mathcal{S}_{\omega}$ is a K\"ahler-Ricci soliton if and only if
$\underline{H}^{}_{g, \Omega} = 0$. Furthermore,
\begin{eqnarray*}
  2 \underline{H}^{}_{g, \Omega}  =  - (\Delta^{\Omega}_{g, J} -
  2\mathbbm{I}) F \in \Lambda^{\Omega, \bot}_{g, J} \cap
  C_{\Omega}^{\infty} \left( X, \mathbbm{R} \right)_0 ,
\end{eqnarray*}
for all $\left( g, \Omega \right) \in \mathcal{S}_{\omega}$. The infinitesimal
properties of the map $\left( g, \Omega \right) \in \mathcal{S}_{\omega}
\longmapsto \underline{H}^{}_{g, \Omega}$ are explained in the next
sub-section.

\subsection{Triple splitting of the space $\mathbbm{T}^J_{g, \Omega}$}

In \cite{Pal2}, we introduce a pseudo-Riemannian metric $G$ over $\mathcal{M}
\times \mathcal{V}_1$ which is positive defined over $\mathbbm{T}^J_{g,
\Omega}$ for any $\left( g, \Omega \right) \in \mathcal{S}_{\omega}$, with $J
\assign - \omega^{- 1} g$. We denote by
\begin{eqnarray*}
  \Lambda^{\Omega, \bot}_{g, J} & \assign & \left[ \tmop{Ker}
  (\Delta^{\Omega}_{g, J} - 2\mathbbm{I}) \right]^{\bot_{\Omega}} \subset
  C_{\Omega}^{\infty} (X, \mathbbm{C})_0,
\end{eqnarray*}
the $L_{\Omega}^2$-orthogonal space to $\Lambda^{\Omega}_{g, J}$ inside
$C_{\Omega}^{\infty} (X, \mathbbm{C})_0$. By abuse of notations we will denote
by $G_{g, \Omega}$ the scalar product over $\Lambda^{\Omega, \bot}_{g, J}$, induced by the isomorphism
\begin{eqnarray*}
  \eta  : \Lambda^{\Omega, \bot}_{g, J} \oplus \mathcal{H}_{g, \Omega}^{0,
  1} \left( T_{X, J} \right) & \longrightarrow & \mathbbm{T}^J_{g, \Omega} \\
  &  & \\
  \left( \psi, A \right) & \longmapsto & \left( g \left(
  \overline{\partial}_{T_{X, J}} \nabla_{g, J}  \overline{\psi} +
  A_{_{_{_{}}}} \right), - \frac{1}{2} \tmop{Re} \left[
  (\Delta^{\Omega}_{g, J} - 2\mathbbm{I}) \psi \right] \Omega \right).
\end{eqnarray*}
Explicitely (see \cite{Pal2}),
\begin{eqnarray*}
  G_{g, \Omega} \left( \varphi, \psi \right) & = & \frac{1}{2} \int_X \left[
  (\Delta^{\Omega}_{g, J} - 2\mathbbm{I}) \varphi_{_{_{_{_{}}}}} \cdot
  \overline{\psi} \; + (\Delta^{\Omega}_{g, J} - 2\mathbbm{I})
  \psi_{_{_{_{_{}}}}} \cdot \overline{\varphi} \right] \Omega\\
  &  & \\
  & + & \frac{1}{2} \int_X \tmop{Im} \left[_{_{_{_{_{}}}}}
  (\Delta^{\Omega}_{g, J} - 2\mathbbm{I}) \varphi \right] \tmop{Im}
  \left[_{_{_{_{_{}}}}} (\Delta^{\Omega}_{g, J} - 2\mathbbm{I}) \psi \right]
  \Omega .
\end{eqnarray*}
For any $(g, \Omega) \in \mathcal{S}_{\omega}$, we introduce in \cite{Pal2}
the vector spaces
\begin{eqnarray*}
  \mathbbm{E}^J_{g, \Omega} & \assign & \left\{ u \in \Lambda^{\Omega,
  \bot}_{g, J} \mid (\Delta^{\Omega}_{g, J} -
  2\mathbbm{I})_{_{_{_{_{_{_{}}}}}}} u = \overline{(\Delta^{\Omega}_{g, J} -
  2\mathbbm{I})_{_{_{_{_{_{_{}}}}}}} u}  \; \right\},\\
  &  & \\
  \mathbbm{O}^J_{g, \Omega} & \assign & \left( \mathbbm{E}^J_{g, \Omega}
  \right)^{\bot_G} \cap \Lambda^{\Omega, \bot}_{g, J},
\end{eqnarray*}
and we denote by $\left[ g, \Omega \right]_{\omega} \assign \tmop{Symp}^0 (X,
\omega) \cdot (g, \Omega) \subset \mathcal{S}_{\omega}$ the orbit of the point
$(g, \Omega)$ under the action of the identity component of the group of
smooth symplectomorphisms $\tmop{Symp}^0 (X, \omega)$ of $X$. The map $\eta$
restricts to a $G$-isometry
\begin{eqnarray*}
  \eta : \mathbbm{O}^J_{g, \Omega} & \longrightarrow & T_{\left[ g, \Omega
  \right]_{\omega}, (g, \Omega)} .
\end{eqnarray*}
The positivity of the metric $G_{g, \Omega}$ over $\Lambda^{\Omega, \bot}_{g,
J}$, combined with an elliptic argument (see \cite{Pal2}) implies the
decomposition
\begin{eqnarray*}
  \Lambda^{\Omega, \bot}_{g, J} & = & \mathbbm{O}^J_{g, \Omega} \oplus_G
  \mathbbm{E}^J_{g, \Omega},
\end{eqnarray*}
Over a compact K\"ahler-Ricci soliton $\left( X, J, g, \omega \right)$, we
introduce the operator
\begin{eqnarray*}
  P^{\Omega}_{g, J} & \assign & (\Delta^{\Omega}_{g, J} - 2\mathbbm{I})
  \overline{ (\Delta^{\Omega}_{g, J} - 2\mathbbm{I})} .
\end{eqnarray*}
This is a non-negative self-adjoint real elliptic operator with respect to
the $L_{\Omega}^2$-hermitian product. The restriction of the differential of
the map map $\left( g, \Omega \right) \in \mathcal{S}_{\omega}
\longmapsto \underline{H}^{}_{g, \Omega}$ over the space $\Lambda^{\Omega, \bot}_{g, J}$,
identifies, via the isomorphism $\eta$, with the map
\begin{eqnarray*}
  D_{g, \Omega}  \underline{H} : \Lambda^{\Omega, \bot}_{g, J} &
  \longrightarrow & \Lambda^{\Omega, \bot}_{g, J} \cap C_{\Omega}^{\infty}
  \left( X, \mathbbm{R} \right)_0\\
  &  & \\
  \psi & \longmapsto & \frac{1}{4} P^{\Omega}_{g, J} \tmop{Re} \psi .
\end{eqnarray*}
This map restricts to an isomorphism 
$$
D_{g, \Omega}  \underline{H}
: \mathbbm{E}^J_{g, \Omega} \longrightarrow \Lambda^{\Omega, \bot}_{g, J} \cap
C_{\Omega}^{\infty} \left( X, \mathbbm{R} \right)_0 ,
$$ 
(see \cite{Pal2} for the technical details), 
and
\begin{eqnarray*}
  \mathbbm{O}^J_{g, \Omega} & = & \tmop{Ker} D_{g, \Omega}  \underline{H} \cap
  \Lambda^{\Omega, \bot}_{g, J} .
\end{eqnarray*}
Moreover, $\tmop{Ker} P^{\Omega}_{g, J} \cap C_{\Omega}^{\infty} \left( X,
\mathbbm{R} \right)_0 = \left\{ \tmop{Re} u \mid u \in \Lambda^{\Omega}_{g, J}
\right\} = : \tmop{Re} \Lambda^{\Omega}_{g, J}$ and
\begin{eqnarray*}
  P^{\Omega}_{g, J} C_{\Omega}^{\infty} \left( X, \mathbbm{R} \right)_0 & = &
  \Lambda^{\Omega, \bot}_{g, J} \cap C_{\Omega}^{\infty} \left( X, \mathbbm{R}
  \right)_0 .
\end{eqnarray*}
In general for any $\left( g, \Omega \right) \in \mathcal{S}_{\omega}$ K\"{a}hler-Ricci soliton
holds the identity
\[ \tmop{Ker} D_{g, \Omega}  \underline{H}^{}_{} \cap \mathbbm{T}^J_{g,
   \Omega} \; = T_{\left[ g, \Omega \right]_{\omega}, (g, \Omega)} \oplus_G
   \mathcal{H}_{g, \Omega}^{0, 1} \left( T_{X, J} \right), \]
with $J : = - \omega^{- 1} g$. We finally notice that applying the finiteness
theorem (see for example \cite{Ebi}, proposition 6.6, page 26), to the real elliptic operator\\
$P^{\Omega}_{g, J} : C_{\Omega}^{\infty} \left( X, \mathbbm{R} \right)_0
\longrightarrow C_{\Omega}^{\infty} \left( X, \mathbbm{R} \right)_0$, we
deduce the $L_{\Omega}^2$-orthogonal decomposition
\begin{equation}
  \label{dec-P-op} C_{\Omega}^{\infty} \left( X, \mathbbm{R} \right)_0 =
  \left[ \Lambda^{\Omega, \bot}_{g, J} \cap C_{\Omega}^{\infty} \left( X,
  \mathbbm{R} \right)_0 \right] \oplus_{_{\Omega}} \tmop{Re}
  \Lambda^{\Omega}_{g, J} .
\end{equation}
\begin{remark}
  \label{tau-rmk}We denote by $\Lambda^{\Omega}_{g, \mathbbm{R}} \assign
  \tmop{Ker}_{\R} (\Delta^{\Omega}_g - 2\mathbbm{I}) \subset C_{\Omega}^{\infty}
  (X, \mathbbm{R})_0$, and by
  \begin{eqnarray*}
    \Lambda^{\Omega, \bot}_{g, \mathbbm{R}} & \assign & \left[ \tmop{Ker}_{\R}
    (\Delta^{\Omega}_g - 2\mathbbm{I}) \right]^{\bot_{\Omega}} \subset
    C_{\Omega}^{\infty} (X, \mathbbm{R})_0,
  \end{eqnarray*}
  its $L_{\Omega}^2$-orthogonal inside $C_{\Omega}^{\infty} (X, R)_0$. It is
  easy to see that the map
  \begin{eqnarray*}
    \chi : \Lambda^{\Omega, \bot}_{g, \mathbbm{R}} \cap C_{\Omega}^{\infty}
    (X, \mathbbm{R})_0 & \longrightarrow & T_{\left[ g, \Omega
    \right]_{\omega}, (g, \Omega)},\\
    &  & \\
    u & \longmapsto & \left( 2 \omega_{_{_{_{_{_{}}}}}}
    \overline{\partial}_{T_{X, J}} \nabla_g u, \left( B^{\Omega}_{g, J} u
    \right) \, \Omega \right),
  \end{eqnarray*}
  is an isomorphism. Thus, there exists an isomorphism map
  \begin{eqnarray*}
    \tau : \mathbbm{O}^J_{g, \Omega} & \longrightarrow & i \Lambda^{\Omega,
    \bot}_{g, \mathbbm{R}} \\
    &  & \\
    \theta & \longmapsto & i u : \, \theta - i u \in
    \Lambda^{\Omega}_{g, J} .
  \end{eqnarray*}
\end{remark}

\section{Variation formulas for the $\Omega$-divergence operators}

For any $u, v \in C^{\infty} \left( X,
S^2 T^{\ast}_X \right)$ we define in \cite{Pal2} the real valued $1$-form
\begin{eqnarray*}
  M_g (u, v) (\xi) & \assign & 2 \nabla_g v (e_k, u^{\ast}_g e_k, \xi) +
  \nabla_g u (\xi, v^{\ast}_g e_k, e_k),
\end{eqnarray*}
for all $\xi \in T_X$.
We show now the following important lemma.
\begin{lemma}
  The first variation of the operator valued map
  $$\left( g, \Omega \right) \mapsto
  \nabla_g^{\ast_{\Omega}} : C^{\infty} \left( X, S^2 T_X \right)
  \longrightarrow C^{\infty} \left( X, T^{\ast}_X \right),
  $$
in arbitrary
  directions $(v, V)$ is given by the formula
  \begin{eqnarray}
    2 \left[ (D_{g, \Omega} \nabla_{\bullet}^{\ast_{\bullet}}) \left( v, V
    \right) \right] u & = & M_g (v, u) - \; 2 u \cdot \left(
    \nabla_g^{\ast_{\Omega}} v^{\ast}_g + \nabla_g V^{\ast}_{\Omega} \right) .
    \label{var-adjDer} 
  \end{eqnarray}
\end{lemma}

\begin{proof}
  We first differentiate the identity defining the covariant derivative of a
  symmetric $2$-tensor $u$ in the direction $v$. We infer
  \begin{eqnarray*}
    \dot{\nabla}_g u (\xi, \eta, \mu) & = & - u \left( \dot{\nabla}_g (\xi,
    \eta), \mu \right) - u \left( \eta, \dot{\nabla}_g (\xi, \mu) \right),
  \end{eqnarray*}
  where $\dot{\nabla}_g \assign (D_g \nabla_{\bullet}) (v)$. Using the
  variation formula for the Levi-Civita connection in \cite{Bes}, we obtain
  \begin{eqnarray*}
    2 \dot{\nabla}_g u (\xi, \eta, \mu) & = & - u \left( \nabla_{g, \xi}
    v^{\ast}_g \eta + \nabla_{g, \eta} v^{\ast}_g \xi - \left( \nabla_g
    v^{\ast}_g \eta \right)_g^T \xi, \mu \right)\\
    &  & \\
    & - & u \left( \eta, \nabla_{g, \xi} v^{\ast}_g \mu + \nabla_{g, \mu}
    v^{\ast}_g \xi - \left( \nabla_g v^{\ast}_g \mu \right)_g^T \xi \right) .
  \end{eqnarray*}
  We transform the term
  \begin{eqnarray*}
    u \left( \left( \nabla_g v^{\ast}_g \eta \right)_g^T \xi, \mu \right) & =
    & g \left( u^{\ast}_g \left( \nabla_g v^{\ast}_g \eta \right)_g^T \xi, \mu
    \right)\\
    &  & \\
    & = & g \left( \left( \nabla_g v^{\ast}_g \eta \right)_g^T \xi,
    u^{\ast}_g \mu \right)\\
    &  & \\
    & = & g \left( \xi, \nabla_g v^{\ast}_g (u^{\ast}_g \mu, \eta) \right)\\
    &  & \\
    & = & \nabla_g v (u^{\ast}_g \mu, \eta, \xi) .
  \end{eqnarray*}
  We deduce the variation formula
  \begin{eqnarray*}
    2 \dot{\nabla}_g u (\xi, \eta, \mu) & = & - u \left( \nabla_{g, \xi}
    v^{\ast}_g \eta + \nabla_{g, \eta} v^{\ast}_g \xi, \mu \right) + \nabla_g
    v (u^{\ast}_g \mu, \xi, \eta)\\
    &  & \\
    & - & u \left( \eta, \nabla_{g, \xi} v^{\ast}_g \mu + \nabla_{g, \mu}
    v^{\ast}_g \xi \right) + \nabla_g v (u^{\ast}_g \eta, \xi, \mu).
  \end{eqnarray*}
Thus, using the fact that $u$ is symmetric we infer
  \begin{eqnarray*}
    2 \left( g^{- 1} \neg \dot{\nabla}_g u \right) \mu & = & 2 u \left(
    \nabla^{\ast}_g v^{\ast}_g, \mu \right) + \nabla_g v (u^{\ast}_g \mu, e_k,
    e_k)\\
    &  & \\
    & - & u \left( \nabla_{g, e_k} v^{\ast}_g \mu + \nabla_{g, \mu}
    v^{\ast}_g e_k, e_k \right) + \nabla_g v (e_k, u^{\ast}_g e_k, \mu),
  \end{eqnarray*}
  where $g^{- 1} \in C^{\infty} (X, S^2 T_X)$ and $(e_k)_k$ is a
  $g$-orthonormed basis of $T_{X, p}$ which diagonalises $u$ at the point $p$.
  We observe however that the right hand-side of the previous equality is
  independent of the choice of the $g$-orthonormed basis $(e_k)_k$ thanks to
  the intrinsic definition of trace. Simplifying, we deduce
  \begin{equation}
    \label{Tr-varCov} 2 \left( g^{- 1} \neg \dot{\nabla}_g u \right) \mu = 2 u
    \left( \nabla^{\ast}_g v^{\ast}_g, \mu \right) + \nabla_g v (u^{\ast}_g
    \mu, e_k, e_k) - \nabla_g v (\mu, u^{\ast}_g e_k, e_k) .
  \end{equation}
  We can compute now the first variation of the expression
  \begin{eqnarray*}
    \nabla_g^{\ast_{\Omega}} u & = & - g^{- 1} \neg \nabla_g u + \nabla_g f
    \neg u,
  \end{eqnarray*}
  with $f \equiv f_{g, \Omega} \assign \log \frac{d V_g}{\Omega}$ . We observe
  the identity
  \begin{eqnarray*}
    \left[ (D_{g, \Omega} \nabla_{\bullet}^{\ast_{\bullet}}) \left( v, V
    \right) \right] u  &=&  (v_g^{\ast} g^{- 1}) \neg \nabla_g u - g^{- 1}
    \neg \dot{\nabla}_g u 
    \\
    \\
    &+&
    \left[ (D_{g, \Omega} \nabla_{\bullet} f_{\bullet,
    \bullet}) \left( v, V \right) \right] \neg u .
  \end{eqnarray*}
  Let $(e_k)_k$ be a $g$-orthonormed local frame of $T_X$ such that $\nabla_g
  e_k  (p) = 0$, for some arbitrary point $p$. Using (\ref{Tr-varCov}) and the
  variation formulas
  \begin{equation}
    \label{var-grad}  \frac{d}{d t}  \left( \nabla_{g_t} f_t \right) =
    \nabla_{g_t}  \dot{f}_t - \dot{g}^{\ast}_t \nabla_{g_t} f_t,
  \end{equation}
  \begin{equation}
    \label{var-f}  \dot{f}_t = \frac{1}{2} \tmop{Tr}_{g_t}  \dot{g}_t -
    \dot{\Omega}^{\ast}_t,
  \end{equation}
  we obtain the equalities at the point $p$,
  \begin{eqnarray*}
    2 \left[ (D_{g, \Omega} \nabla_{\bullet}^{\ast_{\bullet}}) \left( v, V
    \right) \right] u (\mu) & = & 2 \nabla_g u (e_k, v^{\ast}_g e_k, \mu) - 2
    u \left( \nabla^{\ast}_g v^{\ast}_g, \mu \right) \\
    &  & \\
    & - & \nabla_g v (u^{\ast}_g \mu, e_k, e_k) + \nabla_g v (\mu, u^{\ast}_g
    e_k, e_k)\\
    &  & \\
    & + & u \left( \nabla_g \left( \tmop{Tr}_g v - 2 V^{\ast}_{\Omega}
    \right) - 2 v^{\ast}_g \nabla_g f, \mu \right)\\
    &  & \\
    & = & 2 \nabla_g u (e_k, v^{\ast}_g e_k, \mu) - 2 u \left(
    \nabla_g^{\ast_{\Omega}} v^{\ast}_g, \mu \right)\\
    &  & \\
    & - & \nabla_g v (u^{\ast}_g \mu, e_k, e_k) + \nabla_g v (\mu, u^{\ast}_g
    e_k, e_k)\\
    &  & \\
    & + & u \left( \nabla_g \left( \tmop{Tr}_g v - 2 V^{\ast}_{\Omega}
    \right), \mu \right)\\
    &  & \\
    & = & M_g (v, u) (\mu) - \; 2 u \left( \nabla_g^{\ast_{\Omega}}
    v^{\ast}_g + \nabla_g V^{\ast}_{\Omega}, \mu \right),
  \end{eqnarray*}
  thanks to the identity at the point $p$,
  \begin{eqnarray*}
    \nabla_g v (u^{\ast}_g \mu, e_k, e_k) & = & u \left( \nabla_g \tmop{Tr}_g
    v, \mu \right) .
  \end{eqnarray*}
  In order to see this last fact, we observe the equalities
  \begin{eqnarray*}
    u \left( \nabla_g \tmop{Tr}_g v, \mu \right) & = & g \left( u^{\ast}_g
    \nabla_g \tmop{Tr}_g v, \mu \right)\\
    &  & \\
    & = & g \left( \nabla_g \tmop{Tr}_g v, u^{\ast}_g \mu \right)\\
    &  & \\
    & = & \left( d \nabla_g \tmop{Tr}_g v \right)  (u^{\ast}_g \mu)\\
    &  & \\
    & = & (u^{\ast}_g \mu) . v (e_k, e_k)\\
    &  & \\
    & = & \nabla_g v (u^{\ast}_g \mu, e_k, e_k),
  \end{eqnarray*}
  at the point $p$. We obtain the required variation formula.
\end{proof}

In a similar way we compute the first variation formula for the operator
$\tmop{div}_g^{\Omega}$ acting on 1-forms.

\begin{lemma}
  The first variation of the operator valued map
  $$\left( g, \Omega \right) \mapsto
  \tmop{div}_g^{\Omega} : C^{\infty} \left( X, T^{\ast}_X \right)
  \longrightarrow C^{\infty} \left( X, \mathbbm{R} \right),
  $$ 
  in arbitrary
  directions $(v, V)$ is given by the formula
  \begin{eqnarray*}
    \left[ \left( D_{g, \Omega} \tmop{div}^{\bullet}_{\bullet} \right) \left(
    v, V \right)_{_{_{_{}}}} \right] \alpha & = & - \; \langle \nabla_g
    \alpha^{\ast}_g, v^{\ast}_g \rangle_g + \; 2 \alpha \cdot \left(
    \nabla_g^{\ast_{\Omega}} v^{\ast}_g + \nabla_g V^{\ast}_{\Omega} \right) .
  \end{eqnarray*}
\end{lemma}

We include the proof for readers convenience.

\begin{proof}
  Let $\alpha$ be a $1$-form and let $\xi, \eta$ be two smooth vector fields.
  Differentiating the identity
  \begin{eqnarray*}
    \xi . \left( \alpha \cdot \eta \right) & = & \nabla_{g, \xi} \alpha \cdot
    \eta + \alpha \cdot \nabla_{g, \xi} \eta,
  \end{eqnarray*}
  with respect to the variable $g$ we obtain
  \begin{eqnarray*}
    2 \dot{\nabla}_g \alpha \left( \xi, \eta \right) & = & - \alpha \cdot 2
    \dot{\nabla}_g \left( \xi, \eta \right) \\
    &  & \\
    & = & - \alpha \cdot \left( \nabla_{g, \xi} v^{\ast}_g \cdot \eta +
    \nabla_{g, \eta} v^{\ast}_g \cdot \xi \right) + \nabla_g v \left(
    \alpha^{\ast}_g, \xi, \eta \right) .
  \end{eqnarray*}
  We notice indeed the equalities
  \begin{eqnarray*}
    \alpha \cdot \left[ \left( \nabla_{g, \bullet} v^{\ast}_g \cdot \eta
    \right)_g^T \cdot \xi_{_{_{_{}}}} \right] & = & g \left( \alpha^{\ast}_g,
    \left( \nabla_{g, \bullet} v^{\ast}_g \cdot \eta \right)_g^T \cdot
    \xi_{_{_{_{}}}} \right)\\
    &  & \\
    & = & g \left( \nabla_{g, \alpha^{\ast}_g} v^{\ast}_g \cdot \eta,
    \xi_{_{_{_{}}}} \right)\\
    &  & \\
    & = & \nabla_g v \left( \alpha^{\ast}_g, \xi, \eta \right) .
  \end{eqnarray*}
  We deduce
  \begin{eqnarray*}
    2 \left( g^{- 1} \neg \dot{\nabla}_g \alpha \right) & = & 2 \alpha \cdot
    \nabla^{\ast}_g v^{\ast}_g \noplus + \alpha^{\ast}_g . \tmop{Tr}_g v\\
    &  & \\
    & = & \alpha \cdot \left( 2 \nabla^{\ast}_g v^{\ast}_g \noplus + \nabla_g
    \tmop{Tr}_g v \right) .
  \end{eqnarray*}
  We can compute now the first variation of the expression
  \begin{eqnarray*}
    \tmop{div}_g^{\Omega} \alpha & = & g^{- 1} \neg \nabla_g \alpha - \;
    \alpha \cdot \nabla_g f .
  \end{eqnarray*}
  We observe the identities
  \begin{eqnarray*}
    2 \left[ (D_{g, \Omega} \tmop{div}^{\bullet}_{\bullet}) \left( v, V
    \right) \right] \alpha & = & - 2 (v_g^{\ast} g^{- 1}) \neg \nabla_g \alpha
    + 2 g^{- 1} \neg \dot{\nabla}_g \alpha 
    \\
    &  & \\
    &-& 2 \alpha \cdot \left[ (D_{g,
    \Omega} \nabla_{\bullet} f_{\bullet, \bullet}) \left( v, V \right) \right]
    \\
    &  & \\
    & = & - 2 \nabla_g \alpha \left( e_k, v^{\ast}_g e_k \right) + \alpha
    \cdot \left( 2 \nabla^{\ast}_g v^{\ast}_g \noplus + \nabla_g \tmop{Tr}_g v
    \right)\\
    &  & \\
    & - & \alpha \cdot \left( \nabla_g \left( \tmop{Tr}_g v - 2
    V^{\ast}_{\Omega} \right) - 2 v^{\ast}_g \cdot \nabla_g f \right) .
  \end{eqnarray*}
  We infer the required variation formula.
\end{proof}

We can compute now a first variation formula for the double divergence
operator $\tmop{div}_g^{\Omega} \nabla_g^{\ast_{\Omega}}$. 
We observe first
the trivial identity
\begin{eqnarray*}
  \left[ D_{g, \Omega} \left( \tmop{div}^{\bullet}_{\bullet}
  \nabla_{\bullet}^{\ast_{\bullet}} \right) \left( v, V \right)_{_{_{_{}}}}
  \right] v_{_{_{_{_{}}}}} & = & \left[ \left( D_{g, \Omega}
  \tmop{div}^{\bullet}_{\bullet} \right) \left( v, V \right)_{_{_{_{}}}}
  \right] \nabla_g^{\ast_{\Omega}} v_{_{_{_{_{}}}}}\\
  &  & \\
  & + & \tmop{div}_g^{\Omega} \left\{ \left[ \left( D_{g, \Omega}
  \nabla_{\bullet}^{\ast_{\bullet}} \right) \left( v, V \right)_{_{_{_{}}}}
  \right] v_{_{_{_{_{}}}}} \right\},
\end{eqnarray*}
and we explicit the last term;
\begin{eqnarray*}
  &  & 2 \tmop{div}_g^{\Omega} \left\{ \left[ \left( D_{g, \Omega}
  \nabla_{\bullet}^{\ast_{\bullet}} \right) \left( v, V \right)_{_{_{_{}}}}
  \right] v_{_{_{_{_{}}}}} \right\}\\
  &  & \\
  & = & e_l . \left[ 2 \nabla_g v \left( e_k, v^{\ast}_g e_k, e_l \right) +
  \nabla_g v \left( e_l, v^{\ast}_g e_k, e_k \right)_{_{_{_{_{_{}}}}}} - 2 v
  \left( \nabla_g^{\ast_{\Omega}} v^{\ast}_g + V^{\ast}_{\Omega}, e_l \right)
  \right]\\
  &  & \\
  & - & 2 \nabla_g v \left( e_k, v^{\ast}_g e_k, \nabla_g f \right) -
  \nabla_g v \left( \nabla_g f, v^{\ast}_g e_k, e_k \right)_{_{_{_{_{_{}}}}}}
  + 2 v \left( \nabla_g^{\ast_{\Omega}} v^{\ast}_g + V^{\ast}_{\Omega},
  \nabla_g f \right).
  \end{eqnarray*}
  Developing further we obtain
\begin{eqnarray*}
  &  & 2 \tmop{div}_g^{\Omega} \left\{ \left[ \left( D_{g, \Omega}
  \nabla_{\bullet}^{\ast_{\bullet}} \right) \left( v, V \right)_{_{_{_{}}}}
  \right] v_{_{_{_{_{}}}}} \right\}\\
  &  & \\
  & = & 2 g (\nabla_{g, e_l} \nabla_{g, e_k} v_g^{\ast} \cdot v^{\ast}_g e_k,
  e_l) + 2 g (\nabla_{g, e_k} v_g^{\ast} \cdot \nabla_{g, e_l} v^{\ast}_g e_k,
  e_l)\\
  &  & \\
  & + & \nabla^2_{g, e_l, e_l} v \left( v^{\ast}_g e_k, e_k \right) + g
  (\nabla_{g, e_l} v_g^{\ast} \cdot \nabla_{g, e_l} v^{\ast}_g e_k, e_k)\\
  &  & \\
  & - & 2 \nabla_{g, e_l} v \left( e_l, \nabla_g^{\ast_{\Omega}} v^{\ast}_g +
  V^{\ast}_{\Omega} \right) - 2 v \left(_{_{_{_{_{_{}}}}}} \nabla_{g, e_l}
  \left( \nabla_g^{\ast_{\Omega}} v^{\ast}_g + V^{\ast}_{\Omega} \right), e_l
  \right)\\
  &  & \\
  & - & 2 g (\nabla_{g, e_k} v_g^{\ast} \cdot v^{\ast}_g e_k, \nabla_g f) -
  \nabla_g v \left( \nabla_g f, v^{\ast}_g e_k, e_k \right)_{_{_{_{_{_{}}}}}}
  + 2 v \left( \nabla_g^{\ast_{\Omega}} v^{\ast}_g + V^{\ast}_{\Omega},
  \nabla_g f \right)\\
  &  & \\
  & = & 2 g (v^{\ast}_g e_k, \nabla_{g, e_l} \nabla_{g, e_k} v_g^{\ast} e_l)
  + 2 g (\nabla_{g, e_l} v^{\ast}_g e_k, \nabla_{g, e_k} v_g^{\ast} e_l)\\
  &  & \\
  & - & \Delta^{\Omega}_g v \left( v^{\ast}_g e_k, e_k \right) + g
  (\nabla_{g, e_l} v^{\ast}_g e_k, \nabla_{g, e_l} v_g^{\ast} e_k)\\
  &  & \\
  & + & 2 \nabla_g^{\ast_{\Omega}} v_{_{_{_{_{}}}}} \cdot \left(
  \nabla_g^{\ast_{\Omega}} v^{\ast}_g + V^{\ast}_{\Omega} \right) - 2 g
  \left(_{_{_{_{_{_{}}}}}} \nabla_{g, e_l} \left( \nabla_g^{\ast_{\Omega}}
  v^{\ast}_g + V^{\ast}_{\Omega} \right), v^{\ast}_g e_l \right)\\
  &  & \\
  & - & 2 g (v^{\ast}_g e_k, \nabla_{g, e_k} v_g^{\ast} \cdot \nabla_g f) .
\end{eqnarray*}
If we set $$\widehat{\nabla_g v_g^{\ast} } \left( \xi, \eta \right) : =
\nabla_g v_g^{\ast} \left( \eta, \xi \right),$$ then the last expression writes
as
\begin{eqnarray*}
  &  & 2 \tmop{div}_g^{\Omega} \left\{ \left[ \left( D_{g, \Omega}
  \nabla_{\bullet}^{\ast_{\bullet}} \right) \left( v, V \right)_{_{_{_{}}}}
  \right] v_{_{_{_{_{}}}}} \right\}\\
  &  & \\
  & = & 2 g \left( v^{\ast}_g e_k, \nabla_{g, e_l} \widehat{\nabla_g
  v_g^{\ast} } \left( e_l, e_k \right) \right) + 2 g \left( \nabla_g
  v^{\ast}_g  \left( e_l, e_k \right), \widehat{\nabla_g v_g^{\ast} } \left(
  e_l, e_k \right) \right)\\
  &  & \\
  & - & g \left( \Delta^{\Omega}_g v^{\ast}_g \cdot v^{\ast}_g e_k,
  e_{k_{_{_{}}}} \right) + | \nabla_g v^{\ast}_g |^2_g \\
  &  & \\
  & + & 2 \nabla_g^{\ast_{\Omega}} v_{_{_{_{_{}}}}} \cdot \left(
  \nabla_g^{\ast_{\Omega}} v^{\ast}_g + V^{\ast}_{\Omega} \right) - 2
  \left\langle \nabla_g  \left( \nabla_g^{\ast_{\Omega}} v^{\ast}_g +
  V^{\ast}_{\Omega} \right)_{_{_{_{_{}}}}}, v^{\ast}_g \right\rangle_g\\
  &  & \\
  & - & 2 g \left( v^{\ast}_g e_k, \widehat{\nabla_g v_g^{\ast} } \left(
  \nabla_g f, e_k \right) \right) .
\end{eqnarray*}
We infer the formula
\begin{eqnarray*}
  \tmop{div}_g^{\Omega} \left\{ \left[ \left( D_{g, \Omega}
  \nabla_{\bullet}^{\ast_{\bullet}} \right) \left( v, V \right)_{_{_{_{}}}}
  \right] v_{_{_{_{_{}}}}} \right\} & = & - \; \frac{1}{4} \Delta^{\Omega}_g 
  |v|^2_g \\
  &  & \\
  & - & \left\langle \nabla_g^{\ast_{\Omega}} \widehat{\nabla_g v_g^{\ast} },
  v^{\ast}_g \right\rangle_g \noplus + \left\langle \widehat{\nabla_g
  v_g^{\ast} }, \nabla_g v^{\ast}_g \right\rangle_g\\
  &  & \\
  & + & \nabla_g^{\ast_{\Omega}} v \cdot \left( \nabla_g^{\ast_{\Omega}}
  v_g^{\ast} + \nabla_g V^{\ast}_{\Omega} \right)\\
  &  & \\
  & - & \left\langle \nabla_g \left( \nabla_g^{\ast_{\Omega}} v_g^{\ast} +
  \nabla_g V^{\ast}_{\Omega} \right)_{_{_{_{_{}}}}}, v^{\ast}_g
  \right\rangle_g \noplus .
\end{eqnarray*}
We obtain in conclusion the variation identity
\begin{eqnarray}
  \left[ D_{g, \Omega} \left( \tmop{div}^{\bullet}_{\bullet}
  \nabla_{\bullet}^{\ast_{\bullet}} \right) \left( v, V \right)_{_{_{_{}}}}
  \right] v_{_{_{_{_{}}}}} & = & - \; \frac{1}{4} \Delta^{\Omega}_g  |v|^2_g 
  \nonumber\\
  &  &  \nonumber\\
  & - & \left\langle \nabla_g^{\ast_{\Omega}} \widehat{\nabla_g v_g^{\ast} },
  v^{\ast}_g \right\rangle_g \noplus + \left\langle \widehat{\nabla_g
  v_g^{\ast} }, \nabla_g v^{\ast}_g \right\rangle_g \noplus \nonumber\\
  &  &  \nonumber\\
  & + & 2 \nabla_g^{\ast_{\Omega}} v \cdot \left( \nabla_g^{\ast_{\Omega}}
  v_g^{\ast} + \nabla_g V^{\ast}_{\Omega} \right) \nonumber\\
  &  &  \nonumber\\
  & - & \left\langle \nabla_g \left( 2 \nabla_g^{\ast_{\Omega}} v_g^{\ast} +
  \nabla_g V^{\ast}_{\Omega} \right), v^{\ast}_{g_{_{_{_{}}}}} \right\rangle_g
  \noplus . \label{var-div2} 
\end{eqnarray}

\section{The second variation of Perelman's $H$ map}

\begin{lemma}
  The Hessian form $\nabla_G D H (g, \Omega)$ of Perelman's map \\$\left( g, \Omega \right) \in \mathcal{M} \times \mathcal{V}_1
\longmapsto H_{g, \Omega}$, with respect
  to the pseudo-Riemannian structure $G$ at the point $(g, \Omega) \in
  \mathcal{M} \times \mathcal{V}_1$ in arbitrary directions $(v, V)$ is given
  by the expression
  \begin{eqnarray*}
    2 \nabla_G D H (g, \Omega) (v, V ; v, V) & = & - \frac{1}{2}  \left\langle
    \mathcal{L}^{\Omega}_g v, v \right\rangle_g - \Delta^{\Omega}_g \left[
    \frac{1}{4}  \left| v \right|^2_g + (V_{\Omega}^{\ast})^2\right]\\
    &  & \\
    & + & \frac{1}{2}  \left| v \right|^2_g + (V_{\Omega}^{\ast})^2 -
    \frac{1}{2} G_{g, \Omega} \left( v, V ; v, V \right)\\
    &  & \\
    & - & 2 \left| \nabla_g^{\ast_{\Omega}} v_g^{\ast} +_{_{_{_{}}}} \nabla_g
    V^{\ast}_{\Omega} \right|^2_g\\
    &  & \\
    & + & \left\langle \nabla_{g_{}} \left( 2 \nabla_g^{\ast_{\Omega}}
    v_g^{\ast} + 3 \nabla_g V^{\ast}_{\Omega} \right),
    v_g^{\ast} \right\rangle_g \noplus\\
    &  & \\
    & + & \left\langle \nabla_g^{\ast_{\Omega}} v_g^{\ast},
    \nabla_g^{\ast_{\Omega}} v_g^{\ast} + 2_{_{_{_{_{_{}}}}}} \nabla_g
    V^{\ast}_{\Omega} \right\rangle_g\\
    &  & \\
    & + & V^{\ast}_{\Omega} \left(_{_{_{_{_{_{}}}}}} \tmop{div}^{\Omega_{}}
    \nabla_g^{\ast_{\Omega}} v^{\ast}_g + \langle v, h_{g, \Omega}
    \rangle_{g_{}} \right) .
  \end{eqnarray*}
\end{lemma}

\begin{proof}
We consider a smooth curve $(g_t, \Omega_t)_{t \in \mathbbm{R}} \subset
  \mathcal{M} \times \mathcal{V}_1$ with $(g_0, \Omega_0) = (g, \Omega)$ and
  with arbitrary speed $( \dot{g}_0, \dot{\Omega}_0) = (v, V)$. We showe in \cite{Pal2}
  that the $G$-covariant derivative $\nabla_G$ of its speed, in the speed direction, is given by the expressions
  \begin{eqnarray*}
    (\theta_t, \Theta_t) & \equiv & \nabla_G ( \dot{g}_t, \dot{\Omega}_t)  (
    \dot{g}_t, \dot{\Omega}_t),\\
    &  & \\
    \theta_t & \assign & \ddot{g}_t + \dot{g}_t \left( \dot{\Omega}^{\ast}_t -
    \dot{g}_t^{\ast} \right),\\
    &  & \\
    \Theta_t & \assign & \ddot{\Omega}_t + \frac{1}{4}  \left[ \left|
    \dot{g}_t \right|^2_t - 2 ( \dot{\Omega}^{\ast}_t)^2 - G_{g_t, \Omega_t} (
    \dot{g}_t, \dot{\Omega}_t ; \dot{g}_t, \dot{\Omega}_t) \right] \Omega_t .
  \end{eqnarray*}
  Then
  \begin{eqnarray*}
    \nabla_G D H (g_t, \Omega_t)  ( \dot{g}_t, \dot{\Omega}_t ;
    \dot{g}_t, \dot{\Omega}_t)
    & = &  \frac{d^2}{d t^2} H (g_t, \Omega_t) -  D_{g_t, \Omega_t} H
    (\theta_t, \Theta_t).
 \end{eqnarray*}   
  Using the first
  variation formula for $H$ in \cite{Pal2} we obtain the equalities
  \begin{eqnarray*}
    &  & 2 \nabla_G D H (g_t, \Omega_t)  ( \dot{g}_t, \dot{\Omega}_t ;
    \dot{g}_t, \dot{\Omega}_t)\\
    &  & \\
    & = & \frac{d}{d t}  \left[ (\Delta^{\Omega_t}_{g_t} - 2\mathbbm{I})
    \dot{\Omega}_t^{\ast} - \tmop{div}^{\Omega_t}_{g_t} \left(
    \nabla_{g_t}^{\ast_{\Omega_t}}  \dot{g}_t + d \dot{\Omega}_t^{\ast}
    \right)_{_{_{_{_{_{}}}}}} - \langle \dot{g}_t, h_t \rangle_{g_t} \right]
    \\
    &  & \\
    & - & (\Delta^{\Omega_t}_{g_t} - 2\mathbbm{I}) \Theta^{\ast}_t +
    \tmop{div}^{\Omega_t}_{g_t} \left( \nabla_{g_t}^{\ast_{\Omega_t}} \theta_t
    + d \Theta^{\ast}_t \right)_{_{_{_{_{_{}}}}}} + \langle \theta_t, h_t
    \rangle_{g_t}.
\end{eqnarray*}     
Using the identity (\ref{var-div2}) we obtain
\begin{eqnarray*}
    &  & 2 \nabla_G D H (g_t, \Omega_t)  ( \dot{g}_t, \dot{\Omega}_t ;
    \dot{g}_t, \dot{\Omega}_t)\\
    &  & \\
    & = & 2 (\Delta^{\Omega_t}_{g_t} -\mathbbm{I}) \left( \frac{d}{d t} 
    \dot{\Omega}_t^{\ast} - \Theta^{\ast}_t \right)\\
    &  & \\
    & + & 2 \langle \nabla_g d \dot{\Omega}_t^{\ast}, \dot{g}_t \rangle_{g_t}
    - 2 \langle \nabla_{g_t}^{\ast_{\Omega_t}}  \dot{g}^{\ast}_t +
    \nabla_{g_t}  \dot{\Omega}_t^{\ast}, \nabla_{g_t}  \dot{\Omega}_t^{\ast}
    \rangle_{g_t} \\
    &  & \\
    & + & \tmop{div}^{\Omega_t}_{g_t} \nabla_{g_t}^{\ast_{\Omega_t}}  \left(
    \theta_t - \ddot{g}_t \right) \\
    &  & \\
    & + & \frac{1}{4} \Delta^{\Omega_t}_{g_t}  | \dot{g}_t |^2_{g_t} +
    \left\langle \nabla_{g_t}^{\ast_{\Omega_t}} \widehat{\nabla_{g_t} 
    \dot{g}_t^{\ast} }, \dot{g}^{\ast}_t \right\rangle_{g_t} \noplus -
    \left\langle \widehat{\nabla_{g_t}  \dot{g}_t^{\ast} }, \nabla_{g_t} 
    \dot{g}^{\ast}_t \right\rangle_{g_t}\\
    &  & \\
    & - & 2 \nabla_{g_t}^{\ast_{\Omega_t}}  \dot{g}_t \cdot \left(
    \nabla_{g_t}^{\ast_{\Omega_t}}  \dot{g}_t^{\ast} + \nabla_{g_t} 
    \dot{\Omega}_t^{\ast} \right) + \left\langle \nabla_{g_t} \left( 2
    \nabla_{g_t}^{\ast_{\Omega_t}}  \dot{g}_t^{\ast} + \nabla_{g_t} 
    \dot{\Omega}_t^{\ast} \right), \dot{g}^{\ast}_t \right\rangle_{g_t}
    \noplus\\
    &  & \\
    & + & \tmop{Tr}_{\mathbbm{R}} \left[ \left( \theta^{\ast}_t - \frac{d}{d
    t}  \dot{g}_t^{\ast} \right) h_t^{\ast} - \dot{g}_t^{\ast} \left(
    \dot{h}^{\ast}_t - \dot{g}_t^{\ast} h^{\ast}_t \right) \right].
  \end{eqnarray*}
  Rearranging the previous expression, we obtain
  \begin{eqnarray*}
    &  & 2 \nabla_G D H (g_t, \Omega_t)  ( \dot{g}_t, \dot{\Omega}_t ;
    \dot{g}_t, \dot{\Omega}_t)\\
    &  & \\
    & = & 2 (\Delta^{\Omega_t}_{g_t} -\mathbbm{I}) \left( \frac{d}{d t} 
    \dot{\Omega}_t^{\ast} - \Theta^{\ast}_t \right) + \frac{1}{4}
    \Delta^{\Omega_t}_{g_t}  | \dot{g}_t |^2_{g_t}\\
    &  & \\
    & + & \langle \nabla_{g_t} d \dot{\Omega}_t^{\ast}, \dot{g}_t
    \rangle_{g_t} - 2 | \nabla_{g_t}^{\ast_{\Omega_t}}  \dot{g}^{\ast}_t +
    \nabla_{g_t}  \dot{\Omega}_t^{\ast} |_{g_t}^2 \\
    &  & \\
    & + & \tmop{div}^{\Omega_t}_{g_t} \nabla_{g_t}^{\ast_{\Omega_t}}  \left(
    \theta_t - \ddot{g}_t \right) \\
    &  & \\
    & + & \left\langle \nabla_{g_t}^{\ast_{\Omega_t}} \widehat{\nabla_{g_t} 
    \dot{g}_t^{\ast} }, \dot{g}^{\ast}_t \right\rangle_{g_t} \noplus -
    \left\langle \widehat{\nabla_{g_t}  \dot{g}_t^{\ast} }, \nabla_{g_t} 
    \dot{g}^{\ast}_t \right\rangle_{g_t}\\
    &  & \\
    & + & 2 \left\langle \nabla_{g_t} \left( \nabla_{g_t}^{\ast_{\Omega_t}} 
    \dot{g}_t^{\ast} + \nabla_{g_t}  \dot{\Omega}_t^{\ast} \right),
    \dot{g}^{\ast}_t \right\rangle_{g_t} \noplus\\
    &  & \\
    & + & \tmop{Tr}_{\mathbbm{R}} \left[ \left( \theta^{\ast}_t - \frac{d}{d
    t}  \dot{g}_t^{\ast} \right) h_t^{\ast} - \dot{g}_t^{\ast} \left(
    \dot{h}^{\ast}_t - \dot{g}_t^{\ast} h^{\ast}_t \right) \right] .
  \end{eqnarray*}
  Using the expression of $\theta_t$, we develop the term
  \begin{eqnarray*}
    \tmop{div}^{\Omega_t}_{g_t} \nabla_{g_t}^{\ast_{\Omega_t}}  \left(
    \theta_t - \ddot{g}_t \right) & = & \tmop{div}^{\Omega_t}
    \nabla_{g_t}^{\ast_{\Omega_t}}  \left[ \dot{\Omega}_t^{\ast} 
    \dot{g}^{\ast}_t - \left( \dot{g}^{\ast}_t \right)^2 \right] .
  \end{eqnarray*}
  For this purpose we remind a few elementary divergence type identities. For any smooth, function
  $u$, vector field $\xi$ and endomorphism section $A$ of $T_X$ holds the
  identities
  \begin{eqnarray*}
    \nabla_g^{\ast_{\Omega}} \left( u A \right) & = & - \; A \cdot \nabla_g u
    + u \nabla_g^{\ast_{\Omega}} A,\\
    &  & \\
    \tmop{div}^{\Omega} \left( u \xi \right) & = & \left\langle \nabla_g u,
    \xi \right\rangle_g + u \tmop{div}^{\Omega} \xi,\\
    &  & \\
    \nabla_g^{\ast_{\Omega}} A^2 & = & - \tmop{Tr}_g \left( \nabla_g A \cdot A
    \right) + A \nabla_g^{\ast_{\Omega}} A .
  \end{eqnarray*}
Furthermore if $A$ is $g$-symmetric then holds also the formulas
  \begin{equation}
    \label{div-Ev} \tmop{div}^{\Omega} \left( A \cdot \xi \right) = - \;
    \left\langle \nabla_g^{\ast_{\Omega}} A, \xi \right\rangle_g + \langle A,
    \nabla_g \xi \rangle_g,
  \end{equation}
  \begin{equation}
    \label{div-Tr} \tmop{div}^{\Omega} \tmop{Tr}_g \left( \nabla_g A \cdot A
    \right) = - \left\langle \nabla_g^{\ast_{\Omega}} \widehat{\nabla_g A}, A
    \right\rangle_g + \left\langle \widehat{\nabla_g A}, \nabla_g A
    \right\rangle_g .
  \end{equation}
  For readers convinience we show (\ref{div-Ev}) and (\ref{div-Tr}) in the appendix. Using the previous formulas we obtain the
  equalities
  \begin{eqnarray*}
    \tmop{div}^{\Omega_t} \nabla_{g_t}^{\ast_{\Omega_t}}  \left[
    \dot{\Omega}_t^{\ast}  \dot{g}^{\ast}_t - \left( \dot{g}^{\ast}_t
    \right)^2 \right] & = & \tmop{div}^{\Omega_t}  \left[ - \dot{g}^{\ast}_t
    \nabla_{g_t}  \dot{\Omega}_t^{\ast} \noplus + \dot{\Omega}_t^{\ast}
    \nabla_{g_t}^{\ast_{\Omega_t}}  \dot{g}^{\ast}_t \right]\\
    &  & \\
    & + & \tmop{div}^{\Omega_t}  \left[ \tmop{Tr}_{g_t} \left( \nabla_{g_t} 
    \dot{g}_t^{\ast} \cdot \dot{g}^{\ast}_t \right) \noplus - \dot{g}_t^{\ast}
    \nabla_{g_t}^{\ast_{\Omega_t}}  \dot{g}^{\ast}_t \right]\\
    &  & \\
    & = & - \; \langle \nabla_{g_t} d \dot{\Omega}_t^{\ast}, \dot{g}_t
    \rangle_{g_t} + 2 \langle \nabla_{g_t}^{\ast_{\Omega_t}} 
    \dot{g}^{\ast}_t, \nabla_{g_t}  \dot{\Omega}_t^{\ast} \rangle_{g_t}\\
    &  & \\
    & + & \dot{\Omega}_t^{\ast} \tmop{div}^{\Omega_t}
    \nabla_{g_t}^{\ast_{\Omega_t}}  \dot{g}^{\ast}_t\\
    &  & \\
    & - & \left\langle \nabla_{g_t}^{\ast_{\Omega_t}} \widehat{\nabla_{g_t} 
    \dot{g}_t^{\ast} }, \dot{g}^{\ast}_t \right\rangle_{g_t} \noplus +
    \left\langle \widehat{\nabla_{g_t}  \dot{g}_t^{\ast} }, \nabla_{g_t} 
    \dot{g}^{\ast}_t \right\rangle_{g_t}\\
    &  & \\
    & + & | \nabla_{g_t}^{\ast_{\Omega_t}}  \dot{g}^{\ast}_t |_{g_t}^2 -
    \left\langle \dot{g}^{\ast}_t, \nabla_{g_t} \nabla_{g_t}^{\ast_{\Omega_t}}
    \dot{g}_t^{\ast} \right\rangle_{g_t} \noplus\\
    &  & \\
    & = & - \left\langle \nabla_{g_t} \left( \nabla_{g_t}^{\ast_{\Omega_t}} 
    \dot{g}_t^{\ast} + \nabla_{g_t}  \dot{\Omega}_t^{\ast} \right),
    \dot{g}^{\ast}_t \right\rangle_{g_t} \noplus\\
    &  & \\
    & + & \langle \nabla_{g_t}^{\ast_{\Omega_t}}  \dot{g}^{\ast}_t,
    \nabla_{g_t}^{\ast_{\Omega_t}}  \dot{g}^{\ast}_t + 2 \nabla_{g_t} 
    \dot{\Omega}_t^{\ast} \rangle_{g_t} \\
    &  & \\
    & + & \dot{\Omega}_t^{\ast} \tmop{div}^{\Omega_t}
    \nabla_{g_t}^{\ast_{\Omega_t}}  \dot{g}^{\ast}_t\\
    &  & \\
    & - & \left\langle \nabla_{g_t}^{\ast_{\Omega_t}} \widehat{\nabla_{g_t} 
    \dot{g}_t^{\ast} }, \dot{g}^{\ast}_t \right\rangle_{g_t} \noplus +
    \left\langle \widehat{\nabla_{g_t}  \dot{g}_t^{\ast} }, \nabla_{g_t} 
    \dot{g}^{\ast}_t \right\rangle_{g_t} .
  \end{eqnarray*}
  Pluging this identity in the last expression of the Hessian of $H$ we obtain
  \begin{eqnarray*}
    &  & 2 \nabla_G D H (g_t, \Omega_t)  ( \dot{g}_t, \dot{\Omega}_t ;
    \dot{g}_t, \dot{\Omega}_t)\\
    &  & \\
    & = & 2 (\Delta^{\Omega_t}_{g_t} -\mathbbm{I}) \left( \frac{d}{d t} 
    \dot{\Omega}_t^{\ast} - \Theta^{\ast}_t \right) + \frac{1}{4}
    \Delta^{\Omega_t}_{g_t}  | \dot{g}_t |_{g_t}^2\\
    &  & \\
    & + & \langle \nabla_{g_t} d \dot{\Omega}_t^{\ast}, \dot{g}_t
    \rangle_{g_t} - 2 | \nabla_{g_t}^{\ast_{\Omega_t}}  \dot{g}^{\ast}_t +
    \nabla_{g_t}  \dot{\Omega}_t^{\ast} |_{g_t}^2 \\
    &  & \\
    & + & \left\langle \nabla_{g_t} \left( \nabla_{g_t}^{\ast_{\Omega_t}} 
    \dot{g}_t^{\ast} + \nabla_{g_t}  \dot{\Omega}_t^{\ast} \right),
    \dot{g}^{\ast}_t \right\rangle_{g_t} \noplus\\
    &  & \\
    & + & \langle \nabla_{g_t}^{\ast_{\Omega_t}}  \dot{g}^{\ast}_t,
    \nabla_{g_t}^{\ast_{\Omega_t}}  \dot{g}^{\ast}_t + 2 \nabla_{g_t} 
    \dot{\Omega}_t^{\ast} \rangle_{g_t} + \dot{\Omega}_t^{\ast}
    \tmop{div}^{\Omega_t} \nabla_{g_t}^{\ast_{\Omega_t}}  \dot{g}^{\ast}_t\\
    &  & \\
    & + & \dot{\Omega}_t^{\ast} \langle \dot{g}_t, h_t \rangle_{g_t} -
    \frac{1}{2}  \left\langle \mathcal{L}^{\Omega_t}_{g_t}  \dot{g}_t -
    L_{\nabla_{g_t}^{\ast_{\Omega_t}}  \dot{g}_t^{\ast} + \nabla_{g_t} 
    \dot{\Omega}^{\ast}_t} g_t, \dot{g}_t \right\rangle_{g_t},
  \end{eqnarray*}
  thanks to the variation formula of $h$ in \cite{Pal2}. Rearranging the previous expresion, we infer
  \begin{eqnarray*}
    &  & 2 \nabla_G D H (g_t, \Omega_t)  ( \dot{g}_t, \dot{\Omega}_t ;
    \dot{g}_t, \dot{\Omega}_t)\\
    &  & \\
    & = & 2 (\Delta^{\Omega_t}_{g_t} -\mathbbm{I}) \left( \frac{d}{d t} 
    \dot{\Omega}_t^{\ast} - \Theta^{\ast}_t \right) + \frac{1}{4}
    \Delta^{\Omega_t}_{g_t}  | \dot{g}_t |_{g_t}^2\\
    &  & \\
    & - & 2 | \nabla_{g_t}^{\ast_{\Omega_t}}  \dot{g}^{\ast}_t + \nabla_{g_t}
    \dot{\Omega}_t^{\ast} |_{g_t}^2 \\
    &  & \\
    & + & \left\langle \nabla_{g_t} \left( 2 \nabla_{g_t}^{\ast_{\Omega_t}} 
    \dot{g}_t^{\ast} + 3 \nabla_{g_t}  \dot{\Omega}_t^{\ast} \right),
    \dot{g}^{\ast}_t \right\rangle_{g_t} \noplus\\
    &  & \\
    & + & \langle \nabla_{g_t}^{\ast_{\Omega_t}}  \dot{g}^{\ast}_t,
    \nabla_{g_t}^{\ast_{\Omega_t}}  \dot{g}^{\ast}_t + 2 \nabla_{g_t} 
    \dot{\Omega}_t^{\ast} \rangle_{g_t}\\
    &  & \\
    & + & \dot{\Omega}_t^{\ast} \left(_{_{_{_{_{_{}}}}}}
    \tmop{div}^{\Omega_t} \nabla_{g_t}^{\ast_{\Omega_t}}  \dot{g}^{\ast}_t +
    \langle \dot{g}_t, h_t \rangle_{g_t} \right) - \frac{1}{2}  \left\langle
    \mathcal{L}^{\Omega_t}_{g_t}  \dot{g}_t, \dot{g}_t \right\rangle_{g_t} .
  \end{eqnarray*}
  Then the conclusion follows from the expression of $\Theta_t$.
\end{proof}

In \cite{Pal2} we show that the space $G$-orthogonal to the tangent to the
orbit of a point $(g, \Omega) \in \mathcal{M} \times \mathcal{V}_1$, under the action of the identity component of the diffeomorphism group, is
\begin{eqnarray*}
  \mathbbm{F}_{g, \Omega} & \assign & \left\{ (v, V) \in T_{\mathcal{M} \times
  \mathcal{V}_1} \mid \nabla_g^{\ast_{\Omega}} v_g^{\ast} + \nabla_g
  V^{\ast}_{\Omega} = 0 \right\} .
\end{eqnarray*}
\begin{corollary}
  \label{corol-sec-varH}The Hessian form $\nabla_G D H (g, \Omega)$ of Perelman's map \\$(g,\Omega)\in\mathcal{M} \times \mathcal{V}_1\longmapsto H_{g,\Omega}$, with respect to the pseudo-Riemannian structure $G$ at the
  point $(g, \Omega) \in \mathcal{M} \times \mathcal{V}_1$ in arbitrary
  directions $(v, V) \in \mathbbm{F}_{g, \Omega}$, is given by the expression
  \begin{eqnarray*}
   & & 2 \nabla_G D H (g, \Omega) (v, V ; v, V) \\
    &  & \\
   & = & - \frac{1}{2}  \left\langle
    \left( \mathcal{L}^{\Omega}_g + 2 \nabla_{g_{}} \nabla_g^{\ast_{\Omega}}
    \right) v, v \right\rangle_g\\
    &  & \\
    & - & \frac{1}{2}  \left( \Delta^{\Omega}_g - 2\mathbbm{I} \right) \left[
    \frac{1}{2}  \left| v \right|^2_g + (V_{\Omega}^{\ast})^2 - \frac{1}{2}
    G_{g, \Omega} \left( v, V ; v, V \right)\right] \\
    &  & \\
    & + & V^{\ast}_{\Omega} \langle v, h_{g, \Omega} \rangle_{g_{}} \noplus
    \noplus .
  \end{eqnarray*}
\end{corollary}

\section{Application of the weighted Bochner identity}

We observe that the formal adjoint of the $\overline{\partial}_{T_{X, J}}$
operator with respect to the hermitian product
\begin{equation}
  \label{L2omOm-prod}  \left\langle \cdot, \cdot \right\rangle_{\omega,
  \Omega} \assign \int_X \left\langle \cdot, \cdot \right\rangle_{\omega}
  \Omega,
\end{equation}
is the operator
\begin{eqnarray*}
  \overline{\partial}^{\ast_{g, \Omega}}_{T_{X, J}} & : = & e^f
  \overline{\partial}^{\ast_g}_{T_{X, J}} \left( e^{- f} \bullet \right) .
\end{eqnarray*}
With this notation, we define the anti-holomorphic $\Omega$-Hodge-Witten
Laplacian operator acting on $T_X$-valued $q$-forms as
\begin{eqnarray*}
  \Delta^{\Omega, - J}_{T_{X, g}} & \assign & \frac{1}{q} 
  \overline{\partial}_{T_{X, J}} \overline{\partial}^{\ast_{g, \Omega}}_{T_{X,
  J}} + \frac{1}{q + 1}  \overline{\partial}^{\ast_{g, \Omega}}_{T_{X, J}}
  \overline{\partial}_{T_{X, J}},
\end{eqnarray*}
with the usual convention $\infty \cdot 0 = 0$, and the functorial convention on the scalar product in \cite{Pal1}. We will omit the symbol
$\Omega$ in the Hodge-Witten
Laplacian operator, when $\Omega = \tmop{Cst} d V_g$. We define the vector space
\begin{eqnarray*}
  \mathcal{H}_{g, \Omega}^{0, 1} \left( T_{X, J} \right) & \assign &
  \tmop{Ker} \Delta^{\Omega, - J}_{T_{X, g}} \cap C^{\infty} \left( X,
  T^{\ast}_{X, - J} \otimes T_{X, J} \right).
\end{eqnarray*}
It has been showed in \cite{Pal2}, that for any smooth $J$-anti-linear endomorphism
section $A$ of the tangent bundle holds the fundamental Bochner type formula
\begin{equation}
  \label{dec-Lich2} \mathcal{L}^{\Omega}_g A = 2 \Delta^{- J}_{T_{X, g}} A +
  [\tmop{Ric}^{\ast} (g), A] + \nabla_g f \neg \nabla_g A .
\end{equation}
We observe that for bidegree reasons holds the equalities
\begin{eqnarray*}
  \overline{\partial}^{\ast_{g, \Omega}}_{T_{X, J}} A & = &
  \nabla_g^{\ast_{\Omega}} A\\
  &  & \\
  & = & \nabla^{\ast}_g A + A \nabla_g f\\
  &  & \\
  & = & \overline{\partial}^{\ast_g}_{T_{X, J}} A + A \nabla_g f .
\end{eqnarray*}
Using the last equality we obtain the expression
\begin{eqnarray*}
  \overline{\partial}_{T_{X, J}} \overline{\partial}^{\ast_{g, \Omega}}_{T_{X,
  J}} A & = & \overline{\partial}_{T_{X, J}}
  \overline{\partial}^{\ast_g}_{T_{X, J}} A + \nabla^{0, 1}_{g, J} A \nabla_g
  f + A \partial^g_{T_{X, J}} \nabla_g f .
\end{eqnarray*}
We observe indeed
\begin{eqnarray*}
  2 \overline{\partial}_{T_{X, J}}  \left( A \nabla_g f \right) & = & \nabla_g
  \left( A \nabla_g f \right) + J \nabla_{g, J \cdot}  \left( A \nabla_g f
  \right)\\
  &  & \\
  & = & \nabla_g A \nabla_g f + A \nabla^2_g f + J \nabla_{g, J \cdot} A
  \nabla_g f + J A \nabla_{g, J \cdot} \nabla_g f\\
  &  & \\
  & = & 2 \nabla^{0, 1}_{g, J} A \nabla_g f + A \left( \nabla^2_g f - J
  \nabla_{g, J \cdot} \nabla_g f \right)\\
  &  & \\
  & = & 2 \nabla^{0, 1}_{g, J} A \nabla_g f + 2 A \partial^g_{T_{X, J}}
  \nabla_g f .
\end{eqnarray*}
For bidegree reasons holds also the identites
\begin{eqnarray*}
  \frac{1}{2}  \overline{\partial}^{\ast_{g, \Omega}}_{T_{X, J}}
  \overline{\partial}_{T_{X, J}} A & = & \frac{1}{2}
  \nabla^{\ast_{\Omega}}_{T_X, g} \overline{\partial}_{T_{X, J}} A\\
  &  & \\
  & = & \nabla^{\ast_{\Omega}}_g \overline{\partial}_{T_{X, J}} A\\
  &  & \\
  & = & \nabla^{\ast}_g \overline{\partial}_{T_{X, J}} A + \nabla_g f \neg
  \overline{\partial}_{T_{X, J}} A .
  \end{eqnarray*}
  Thus
\begin{eqnarray*}
  \frac{1}{2}  \overline{\partial}^{\ast_{g, \Omega}}_{T_{X, J}}
  \overline{\partial}_{T_{X, J}} A 
  & = & \frac{1}{2} \nabla^{\ast}_{T_X, g} \overline{\partial}_{T_{X, J}} A +
  \nabla_g f \neg \overline{\partial}_{T_{X, J}} A\\
  &  & \\
  & = & \frac{1}{2}  \overline{\partial}^{\ast_g}_{T_{X, J}}
  \overline{\partial}_{T_{X, J}} A + \nabla_g f \neg \nabla^{0, 1}_{g, J} A -
  \nabla^{0, 1}_{g, J} A \nabla_g f .
\end{eqnarray*}
Combining the identities obtained so far we deduce the expression
\begin{equation}
  \label{Om-AntHol-Hdg-Lap} \Delta^{\Omega, - J}_{T_{X, g}} A = \Delta^{-
  J}_{T_{X, g}} A + \nabla_g f \neg \nabla^{0, 1}_{g, J} A + A
  \partial^g_{T_{X, J}} \nabla_g f .
\end{equation}
Plugging this in the fundamental identity (\ref{dec-Lich2}) we obtain the
equalities
\begin{eqnarray*}
  \mathcal{L}^{\Omega}_g A & = & 2 \Delta^{\Omega, - J}_{T_{X, g}} A +
  [\tmop{Ric}^{\ast} (g), A] - 2 A \partial^g_{T_{X, J}} \nabla_g f\\
  &  & \\
  & - & \nabla_g f \neg \left( \nabla^{0, 1}_{g, J} - \nabla^{1, 0}_{g, J}
  \right) A\\
  &  & \\
  & = & 2 \Delta^{\Omega, - J}_{T_{X, g}} A + [\tmop{Ric}^{\ast} (g), A] - 2
  A \partial^g_{T_{X, J}} \nabla_g f\\
  &  & \\
  & - & (J \nabla_g f) \neg J \nabla_g A ._{}
\end{eqnarray*}
Thus, if $A \in \mathcal{H}_{g, \Omega}^{0, 1} \left( T_{X, J} \right)$, then
holds the stability identity
\begin{equation}
  \label{stab-harm}  \left\langle \mathcal{L}^{\Omega}_g A, A \right\rangle_g
  = - 2 \left\langle \nabla^2_g f, A^2 \right\rangle_g + \left\langle J
  \nabla_g f \neg \nabla_g A, J A \right\rangle_g .
\end{equation}

\section{Variations of $\omega$-compatible complex structures}

Let $\left( X, J, g, \omega \right)$ be a Fano manifold such that $\omega =
\tmop{Ric}_{_J} \left( \Omega \right)$, with $\Omega \in \mathcal{V}_1$ and
let $\left( J_t \right)_t \subset \mathcal{J}_{\omega}$ be a smooth curve such
that $J_0 = J$. We differentiate the definition $g_t \assign - \omega J_t$. We
obtain $\dot{g}^{\ast}_t = - J_t  \dot{J}_t$ and $\ddot{g}^{\ast}_t = - J_t 
\ddot{J}_t$. On the other hand, deriving twice the condition $J^2_t =
-\mathbbm{I}$, we obtain $- ( J_t  \ddot{J}_t)_{J_t}^{1, 0} =
\dot{J}_t^2$ and thus $\left( \ddot{g}^{\ast}_t \right)_{J_t}^{1, 0} = \left(
\dot{g}^{\ast}_t \right)^2$. The latter gives
\[ \left( \ddot{g}^{\ast}_t \right)_{J_t}^{0, 1} = \ddot{g}_t^{\ast} - \left(
   \dot{g}^{\ast}_t \right)^2_t = \frac{d}{d t}  \dot{g}^{\ast}_t . \]
For any endomorphism $A$ of the tangent bundle and for any bilinear form $B$
over it we define the contraction operation $A \neg B \assign \tmop{Alt}
\left( B \circ A \right)$, where $\tmop{Alt}$ is the alternating operator and
the composition operator $\circ$ act on the first entry of $B$. Let $N_J$ be
the Nijenhuis tensor of an arbitrary almost complex structure $J$. Then the
general formula
\begin{eqnarray*}
  \frac{d}{d t} N_{J_t} & = & \dot{J}_t \neg N_{J_t} - \dot{J}_t N_{J_t}
  \noplus + \overline{\partial}_{T_{X, J_t}}  \dot{J}_t,
\end{eqnarray*}
(see a computation in \cite{Pal1}), implies $\overline{\partial}_{T_{X, J_t}} 
\dot{J}_t \equiv 0$, in our case. Thus time deriving the identity
\begin{eqnarray*}
  \frac{d}{d t}  \overline{\partial}_{T_{X, J_t}}  \dot{g}^{\ast}_t & \equiv &
  0,
\end{eqnarray*}
we obtain the property
\begin{equation}
  \overline{\partial}_{T_{X, J_t}}  \frac{d}{d t}  \dot{g}^{\ast}_t =
  \dot{g}^{\ast}_t \neg \nabla^{1, 0}_{g_t, J_t}  \dot{g}^{\ast}_t .
  \label{sec-ord-Defm}
\end{equation}
Indeed we prove the variation formula
\begin{eqnarray*}
  \left( \frac{d}{d t}  \overline{\partial}_{T_{X, J_t}} \right) 
  \dot{g}^{\ast}_t & = & - \dot{g}^{\ast}_t \neg \nabla^{1, 0}_{g_t, J_t} 
  \dot{g}^{\ast}_t .
\end{eqnarray*}
For this purpose we expand the derivative of $\overline{\partial}_{T_{X,
J_t}}$ acting on a smooth endomorphism section $A$ of $T_X$. We obtain
\begin{eqnarray*}
  2 \left[ \left( \frac{d}{d t}  \overline{\partial}_{T_{X, J_t}} \right) A
  \right] \left( \xi, \eta \right) & = & 2 \frac{d}{d t} \tmop{Alt} \left[
  \nabla^{0, 1}_{g_t, J_t} A \left( \xi, \eta \right) \right]\\
  &  & \\
  & = & \tmop{Alt} \frac{d}{d t}  \left[ \nabla_{g_t} A \left( \xi, \eta
  \right) + J_t \nabla_{g_t} A \left( J_t \xi, \eta \right) \right]\\
  &  & \\
  & = & \tmop{Alt} \left[ \dot{\nabla}_{g_t} A \left( \xi, \eta \right) +
  \dot{J}_t \nabla_{g_t} A \left( J_t \xi, \eta \right) \right]\\
  &  & \\
  & + & \tmop{Alt} \left[ J_t \dot{\nabla}_{g_t} A \left( J_t \xi, \eta
  \right) + J_t \nabla_{g_t} A \left( \dot{J}_t \xi, \eta \right) \right] .
\end{eqnarray*}
Using the variation formula 
$$
\dot{\nabla}_{g_t} A \left( \xi, \eta \right) =
\dot{\nabla}_{g_t} \left( \xi, A \eta \right) - A \dot{\nabla}_{g_t} \left(
\xi, \eta \right),
$$ 
and the fact that the bilinear form $\dot{\nabla}_{g_t}$ is
symmetric we deduce the formula
\begin{eqnarray*}
 & & 2 \left[ \left( \frac{d}{d t}  \overline{\partial}_{T_{X, J_t}} \right) A
  \right] \left( \xi, \eta \right) \\
  & &\\
  & = & \tmop{Alt} \left[ \dot{\nabla}_{g_t}
  \left( \xi, A \eta \right) + \dot{J}_t \nabla_{g_t} A \left( J_t \xi, \eta
  \right) \right]\\
  &  & \\
  & + & \tmop{Alt} \left[ J_t \dot{\nabla}_{g_t} \left( J_t \xi, A \eta
  \right) - J_t A \dot{\nabla}_{g_t} \left( J_t \xi, \eta \right) + J_t
  \nabla_{g_t} A \left( \dot{J}_t \xi, \eta \right) \right] .
\end{eqnarray*}
We remind now (see \cite{Pal1}), that time deriving the K\"{a}hler condition
$\nabla_{g_t} J_t \equiv 0$, we obtain the identity
\[ \dot{\nabla}_{g_t} (\eta, \xi) \hspace{0.75em} + \hspace{0.75em} J_t
   \dot{\nabla}_{g_t} (J_t \eta, \xi) \hspace{0.75em} + \hspace{0.75em} J_t
   \nabla_{g_t}  \dot{J}_t (\xi, \eta) \hspace{0.75em} = \hspace{0.75em} 0
   \hspace{0.25em}, \]
Using this in the previous formula with $A = \dot{g}^{\ast}_t = - J_t 
\dot{J}_t$ we obtain
\begin{eqnarray*}
& &  2 \left[ \left( \frac{d}{d t}  \overline{\partial}_{T_{X, J_t}} \right) 
  \dot{g}^{\ast}_t \right] \left( \xi, \eta \right) \\
  &  & \\
  & = & \tmop{Alt} \left[ -
  J_t \nabla_{g_t}  \dot{J}_t \left( \dot{g}^{\ast}_t \eta, \xi \right) -
  \dot{g}^{\ast}_t J_t \nabla_{g_t}  \dot{J}_t (\eta, \xi) \right]\\
  &  & \\
  & + & \tmop{Alt} \left[ \dot{J}_t \nabla_{g_t}  \dot{g}^{\ast}_t \left( J_t
  \xi, \eta \right) + J_t \nabla_{g_t}  \dot{g}^{\ast}_t \left( \dot{J}_t \xi,
  \eta \right) \right]\\
  &  & \\
  & = & \tmop{Alt} \left[ \nabla_{g_t}  \dot{g}^{\ast}_t \left(
  \dot{g}^{\ast}_t \eta, \xi \right) + \dot{J}_t \nabla_{g_t} 
  \dot{g}^{\ast}_t \left( J_t \xi, \eta \right) + J_t \nabla_{g_t} 
  \dot{g}^{\ast}_t \left( \dot{J}_t \xi, \eta \right) \right]\\
  &  & \\
  & - & \dot{g}^{\ast}_t \partial^{g_t}_{T_{X, J_t}}  \dot{g}^{\ast}_t \left(
  \xi, \eta \right)\\
  &  & \\
  & = & \tmop{Alt} \left[ \nabla^{1, 0}_{g_t, J_t}  \dot{g}^{\ast}_t \left(
  \dot{g}^{\ast}_t \eta, \xi \right) - \nabla^{1, 0}_{g_t}  \dot{g}^{\ast}_t
  \left( \dot{g}^{\ast}_t \xi, \eta \right) \right]\\
  &  & \\
  & = & - 2 \left[ \dot{g}^{\ast}_t \neg \nabla^{1, 0}_{g_t, J_t} 
  \dot{g}^{\ast}_t \right] \left( \xi, \eta \right),
\end{eqnarray*}
and thus the required formula. The latter can also be obtained deriving the
Maurer-Cartan equation, which writes in the K\"ahler case (see the appendix)
as
\begin{eqnarray*}
  \overline{\partial}_{T_{X, J}} \mu_t + \mu_t \neg \nabla^{1, 0}_{g, J} \mu_t
  & = & 0,
\end{eqnarray*}
with $\mu_t$ the Caley transform of $J_t$ with respect to $J$.

We remind now (see \cite{Pal2}), that for any smooth family $\left( g_t, \Omega_t \right)_t
\subset \mathcal{S}_{\omega}$, holds the identity
\begin{eqnarray*}
  \Delta^{\Omega_t, - J_t}_{T_{X, g_t}}  \dot{g}^{\ast}_t & = & \left(
  \Delta^{\Omega_t, - J_t}_{T_{X, g_t}}  \dot{g}^{\ast}_t \right)_{g_t}^T,
\end{eqnarray*}
with $J_t \assign - \omega^{- 1} g_t$. The latter rewrites as
\begin{equation}
  \label{basic-kuranishSym}  \overline{\partial}_{T_{X, J_t}}
  \nabla_{g_t}^{\ast_{\Omega_t}}  \dot{g}^{\ast}_t = \left(
  \overline{\partial}_{T_{X, J_t}} \nabla_{g_t}^{\ast_{\Omega_t}} 
  \dot{g}^{\ast}_{t_{_{}}} \right)_{g_t}^T .
\end{equation}
\begin{lemma}
  For any smooth family $\left( g_t, \Omega_t \right)_t \subset
  \mathcal{S}_{\omega}$, with $\left( g, \Omega \right) = \left( g_0, \Omega_0
  \right)$ and $( \dot{g}_0, \dot{\Omega}_0) \in
  \mathbbm{F}^J_{g, \Omega}$, holds the symmetry property
  \begin{equation}
    \label{first-kur-sm}  \overline{\partial}_{T_{X, J_t}}
    \nabla_{g_t}^{\ast_{\Omega_t}}  \frac{d}{d t} _{\mid_{t = 0}} 
    \dot{g}^{\ast}_t = \left( \overline{\partial}_{T_{X, J_t}}
    \nabla_{g_t}^{\ast_{\Omega_t}}  \frac{d}{d t} _{\mid_{t = 0}} 
    \dot{g}^{\ast}_{t_{_{}}} \right)_g^T + \left[ \partial^g_{T_{X, J}}
    \nabla_g^{\ast_{\Omega}}  \dot{g}^{\ast}_0, \dot{g}^{\ast}_0 \right] .
  \end{equation}
\end{lemma}

\begin{proof}
  Let $A$ be a smoth $g$-symmetric endomorphism section of $T_X$.
  Differentiating in the variables $\left( g, \Omega \right)$ the trivial
  identity $\nabla_{g_{}}^{\ast_{\Omega}} A = g^{- 1} \nabla_g^{\ast_{\Omega}}
  \left( g A \right)$, we obtain
  \begin{eqnarray*}
    \left[ \left( D_{g, \Omega} \nabla_{\bullet}^{\ast_{\bullet}} \right)
    \left( v, V \right)_{_{_{_{}}}} \right] A & = & - v^{\ast}_g
    \nabla_{g_{}}^{\ast_{\Omega}} A \noplus + g^{- 1} \left[ \left( D_{g,
    \Omega} \nabla_{\bullet}^{\ast_{\bullet}} \right) \left( v, V
    \right)_{_{_{_{}}}} \right] \left( g A \right) \\
    &  & \\    
    &+& \nabla_g^{\ast_{\Omega}}
    \left( v^{\ast}_g A \right) .
  \end{eqnarray*}
  We observe now the identities
  \begin{eqnarray*}
    M_g (v, v) & = & 2 g \tmop{Tr}_g  \left( \nabla_g v^{\ast}_g \cdot
    v^{\ast}_g \right) + \frac{1}{2} d |v|^2_g \\
    &  & \\
    & = & 2 v \nabla_g^{\ast_{\Omega}} v^{\ast}_g - 2 g
    \nabla_g^{\ast_{\Omega}} \left( v^{\ast}_g \right)^2 + \frac{1}{2} d
    |v|^2_g .
  \end{eqnarray*}
  Then using the variation formula (\ref{var-adjDer}) we infer the fundamental
  identity
  \begin{eqnarray}
    2 \left[ \left( D_{g, \Omega} \nabla_{\bullet}^{\ast_{\bullet}} \right)
    \left( v, V \right)_{_{_{_{}}}} \right] v^{\ast}_g & = & \frac{1}{2}
    \nabla_g |v|^2_g - 2 v^{\ast}_g \cdot \left( \nabla_g^{\ast_{\Omega}}
    v^{\ast}_g + \nabla_g V^{\ast}_{\Omega} \right) . \label{super-var-Div} 
  \end{eqnarray}
  The variation formula for the $\overline{\partial}_{T_{X, J_t}}$-operator
  acting on vector fields in lemma 1 of \cite{Pal3} writes as
  \begin{eqnarray*}
    2 \frac{d}{d t}  \left( \overline{\partial}_{T_{X, J_t}} \xi \right) & = &
    \xi \neg \nabla_{g_t}  \dot{g}^{\ast}_t - \left[ \partial^{g_t}_{T_{X,
    J_t}} \xi, \dot{g}^{\ast}_t \right] + \left[ \overline{\partial}_{T_{X,
    J_t}} \xi, \dot{g}^{\ast}_t \right] .
  \end{eqnarray*}
  Using this, the variation formula (\ref{super-var-Div}) and the assumption
  on the initial speed of the curve $\left( g_t, \Omega_t \right)$, we infer
  \begin{eqnarray*}
    2 \frac{d}{d t} _{\mid_{t = 0}} \left( \overline{\partial}_{T_{X, J_t}}
    \nabla_{g_t}^{\ast_{\Omega_t}}  \dot{g}^{\ast}_{t_{_{}}} \right) & = &
    \nabla_g^{\ast_{\Omega}}  \dot{g}^{\ast}_0 \neg \nabla_g  \dot{g}^{\ast}_0
    \\
    &  & \\
    & - & \left[ \partial^g_{T_{X, J}} \nabla_g^{\ast_{\Omega}} 
    \dot{g}^{\ast}_0, \dot{g}^{\ast}_0 \right] + \left[
    \overline{\partial}_{T_{X, J}} \nabla_g^{\ast_{\Omega}} 
    \dot{g}^{\ast}_{0_{_{}}}, \dot{g}^{\ast}_0 \right]\\
    &  & \\
    & + & \frac{1}{2}  \overline{\partial}_{T_{X, J}} \nabla_g | \dot{g}_0
    |^2_g \noplus + 2 \overline{\partial}_{T_{X, J}} \nabla_g^{\ast_{\Omega}} 
    \frac{d}{d t} _{\mid_{t = 0}}  \dot{g}^{\ast}_t .
  \end{eqnarray*}
  Using this equality, the elementary identity
  \begin{eqnarray*}
    \frac{d}{d t} A^T_{g_t} & = & [A_{g_t}^T, \dot{g}^{\ast}_t],
  \end{eqnarray*}
  for arbitrary endomorphism section $A$ of $T_X$ and time deriving the
  identity (\ref{basic-kuranishSym}), we obtain the required conclusion.
  (Notice that the endomorphism section $\partial^g_{T_{X, J}}
  \nabla_g^{\ast_{\Omega}}  \dot{g}^{\ast}_0$ is $g$-symmetric thanks to the
  assumption $\nabla_g^{\ast_{\Omega}}  \dot{g}^{\ast}_0 = - \nabla_g 
  \dot{\Omega}^{\ast}_0$). 
\end{proof}

\begin{corollary}
  \label{fund-cx-def-sm}Let $\left( J_t \right)_t \subset
  \mathcal{J}_{\omega}$ be a smooth curve such that $\dot{J}_0 \in
  \mathcal{H}_{g, \Omega}^{0, 1} \left( T_{X, J} \right)$ then
  \begin{eqnarray*}
    \nabla_g^{\ast_{\Omega}} \left( \dot{J}_0 \neg \nabla^{1, 0}_{g, J} 
    \dot{J}_0  \right) & = & \left[ \nabla_g^{\ast_{\Omega}} \left( \dot{J}_0
    \neg \nabla^{1, 0}_{g, J}  \dot{J}_0  \right) \right]_g^T .
  \end{eqnarray*}
\end{corollary}

\begin{proof}
  The identity (\ref{sec-ord-Defm}) implies
  \begin{eqnarray*}
    \overline{\partial}_{T_{X, J}} \nabla_g^{\ast_{\Omega}}  \frac{d}{d t}
    _{\mid_{t = 0}}  \dot{g}^{\ast}_t & = & \Delta^{\Omega, - J}_{T_{X, g}} 
    \frac{d}{d t} _{\mid_{t = 0}}  \dot{g}^{\ast}_t - \nabla_g^{\ast_{\Omega}}
    \left( \dot{g}^{\ast}_0 \neg \nabla^{1, 0}_{g, J}  \dot{g}^{\ast}_0 
    \right) .
  \end{eqnarray*}
  Pluging this in the equality (\ref{first-kur-sm}) and using the fact that
  the Laplacian term is $g$-symmetric (see \cite{Pal2}), we infer the required
  conclusion.
\end{proof}

\begin{lemma}
  \label{sec-harm-var}Let $\left( X, J, g, \omega \right)$ be a Fano manifold
  such that $\omega = \tmop{Ric}_{_J} \left( \Omega \right)$, with $\Omega \in
  \mathcal{V}_1$ and let $\left( J_t \right)_t \subset \mathcal{J}_{\omega}$
  be a smooth curve such that $J_0 = J$ and $\nabla_g^{\ast_{\Omega}} 
  \dot{J}_0 = 0$. Then there exists unique $\left( \psi, A_1 \right) \in
  \Lambda^{\Omega, \bot}_{g, J} \oplus \mathcal{H}_{g, \Omega}^{0, 1} (T_{X,
  J})$ such that
  \begin{eqnarray*}
    \frac{d}{d t} _{\mid_{t = 0}}  \dot{g}^{\ast}_t + \nabla_g^{\ast_{\Omega}}
    \left( \Delta^{\Omega, - J}_{T_{X, g}} \right)^{- 1} \left( \dot{J}_0 \neg
    \nabla^{1, 0}_{g, J}  \dot{J}_0  \right) & = & \overline{\partial}_{T_{X,
    J}} \nabla_{g, J}  \overline{\psi} + A_1 .
  \end{eqnarray*}
\end{lemma}

\begin{proof}
  The identity (\ref{sec-ord-Defm}) implies
  \begin{eqnarray*}
    \overline{\partial}_{T_{X, J}} \left[  \frac{d}{d t} _{\mid_{t = 0}} 
    \dot{g}^{\ast}_t + \nabla_g^{\ast_{\Omega}} \left( \Delta^{\Omega, -
    J}_{T_{X, g}} \right)^{- 1} \left( \dot{J}_0 \neg \nabla^{1, 0}_{g, J} 
    \dot{J}_0  \right) \right] & = & 0 .
  \end{eqnarray*}
  Moreover the endomorphism
  \begin{eqnarray*}
    &  & \frac{d}{d t} _{\mid_{t = 0}}  \dot{g}^{\ast}_t +
    \nabla_g^{\ast_{\Omega}} \left( \Delta^{\Omega, - J}_{T_{X, g}} \right)^{-
    1} \left( \dot{J}_0 \neg \nabla^{1, 0}_{g, J}  \dot{J}_0  \right)\\
    &  & \\
    & = &  \frac{d}{d t} _{\mid_{t = 0}}  \dot{g}^{\ast}_t + \left(
    \Delta^{\Omega, - J}_{T_{X, g}} \right)^{- 1} \nabla_g^{\ast_{\Omega}}
    \left( \dot{J}_0 \neg \nabla^{1, 0}_{g, J}  \dot{J}_0  \right),
  \end{eqnarray*}
  is $g$-symmetric thanks to corollary \ref{fund-cx-def-sm}, lemma 13 in \cite{Pal2} and identity (14.7) in \cite{Pal2}. By corollary 3 in \cite{Pal2}, we infer
  the required conclusion. Notice that $\left( \psi, A_1 \right)$ is uniquely
  determined by $\dot{J}_0$ and $\ddot{J}_0$. 
\end{proof}

\section{Proof of theorem \ref{main-teo}}

For any smooth family $\left( g_t, \Omega_t \right)_t \subset
\mathcal{S}_{\omega}$, with $\left( g_0, \Omega_0 \right) = \left( g, \Omega
\right)$, we consider the smooth curve $t \mapsto \gamma_t \assign
\underline{H}_{g_t, \Omega_t} \Omega_t / \Omega \in C_{\Omega}^{\infty} \left(
X, \mathbbm{R} \right)_0$. Then $\left( g_t, \Omega_t \right)_t \equiv \left(
J_t, \omega \right)_t$ is a family of K\"ahler-Ricci solitons if and only if
$\gamma_t \equiv 0$. We assume this identity and we notice that $0 =
\dot{\gamma}_0 = D_{g, \Omega}  \underline{H}( \dot{g}_0,
\dot{\Omega}_0)$. We write
\begin{eqnarray*}
  \dot{g}^{\ast}_0  =  - J \dot{J}_0 = \overline{\partial}_{T_{X, J}}
  \nabla_{g, J} \overline{\theta} + \; 2 A,
\end{eqnarray*}
with $\left( \theta, A \right) \in \Lambda^{\Omega, \bot}_{g, J} \oplus
\mathcal{H}_{g, \Omega}^{0, 1} (T_{X, J})$. The properties of the first
variation of $\underline{H}$ imply $\theta \in \mathbbm{O}^J_{g, \Omega}$.
According to the isomorphism $\tau$ in remark \ref{tau-rmk}, we pick the unique
$u \in \Lambda^{\Omega, \bot}_{g, \mathbbm{R}}$ such that $\theta - i u \in
\Lambda^{\Omega}_{g, J}$ and we consider the one parameter subgroup of
$\omega$-symplectomorphisms $\left( \Psi_t \right)_t$, $\Psi_0 = \tmop{id}_X$,
given by $2 \dot{\Psi}_t = - \left( \omega^{- 1} d u \right) \circ \Psi_t$.
Then $\left( \Psi^{\ast}_t J_t, \omega \right)_t$ is still a family of
K\"ahler-Ricci solitons and
\begin{eqnarray*}
  \frac{d}{d t} _{\mid_{t = 0}} \Psi_t^{\ast} J_t & = & \dot{J}_0 -
  \frac{1}{2} L_{\omega^{- 1} d u} J\\
  &  & \\
  & = & J \overline{\partial}_{T_{X, J}} \nabla_{g, J}  \overline{\left(
  \theta - i u \right)} + 2 J A\\
  &  & \\
  & = & 2 J A .
\end{eqnarray*}
Thus we can assume, without loss of generality in the statement of the theorem
\ref{main-teo}, that the family of K\"ahler-Ricci solitons $\left( J_t, \omega
\right)_t$ satisfies $\dot{J}_0 \in \mathcal{H}_{g, \Omega}^{0, 1} (T_{X,
J})$. Using this assumption, we explicit the second variation of the map
$(g,\Omega)\longmapsto\underline{H}_{g,\Omega}$. The fact that $\dot{g}_0^{\ast} = 2 A$, implies
$\dot{\Omega}_0 = 0$, thanks to the equations defining the space
$\mathbbm{T}^J_{g, \Omega}$. Thus
\begin{eqnarray*}
  2 \frac{d^2}{d t^2} _{\mid_{t = 0}}  \underline{H}_{g_t, \Omega_t} & = & 2
  \nabla_G D \underline{H} (g, \Omega)  ( \dot{g}_0, 0 ; \dot{g}_0, 0) + 2
  D_{g, \Omega}  \underline{H}  \left( \xi, \Xi \right),
\end{eqnarray*}
with
\begin{eqnarray*}
  \xi^{\ast}_g & : = & \frac{d}{d t} _{\mid_{t = 0}}  \dot{g}^{\ast}_t,\\
  &  & \\
  \Xi^{\ast}_{\Omega} & : = & \frac{d}{d t} _{\mid_{t = 0}} 
  \dot{\Omega}^{\ast}_t + \frac{1}{4}  \left| \dot{g}_0 \right|^2_g -
  \frac{1}{4} G_{g, \Omega} ( \dot{g}_0, 0 ; \dot{g}_0, 0) .
\end{eqnarray*}
Using the fact that $\left( g, \Omega \right)$ is a soliton and the first and
second variation formulas for Perelman's functions $H$ (in \cite{Pal2} and corollary
\ref{corol-sec-varH}), and $\mathcal{W}$ (in \cite{Pal2}), we infer
\begin{eqnarray*}
  2 \frac{d^2}{d t^2} _{\mid_{t = 0}}  \underline{H}_{g_t, \Omega_t} & = &
  \nabla_G D \left( 2 H -\mathcal{W} \right) (g, \Omega)  ( \dot{g}_0, 0 ;
  \dot{g}_0, 0) + 2 D_{g, \Omega} H \left( \xi, \Xi \right) \\
  &  & \\
  & = & - 2 \left\langle \mathcal{L}^{\Omega}_g A, A \right\rangle_g -
  (\Delta^{\Omega}_g - 2\mathbbm{I}) \left| A \right|_g^2 - 2 \int_X \left| A
  \right|_g^2 \Omega  \\
  &  & \\
  & + & 2 \int_X \left| A \right|_g^2 F \,\Omega+2 (\Delta^{\Omega}_g -\mathbbm{I}) \Xi^{\ast}_{\Omega} -
  \tmop{div}^{\Omega} \nabla_g^{\ast_{\Omega}} \xi_g^{\ast}\\
  &  & \\
  & = & 2 \int_X \left| A \right|_g^2 F \Omega - 2 \left\langle
  \mathcal{L}^{\Omega}_g A, A \right\rangle_g + \Delta^{\Omega}_g \left| A
  \right|^2_g  \\
  &  & \\
  & + & 2 (\Delta^{\Omega}_g -\mathbbm{I})  \frac{d}{d t} _{\mid_{t =
  0}}  \dot{\Omega}^{\ast}_t-\tmop{div}^{\Omega} \nabla_g^{\ast_{\Omega}}  \frac{d}{d t} _{\mid_{t
  = 0}}  \dot{g}^{\ast}_t .
\end{eqnarray*}
Using lemma \ref{sec-harm-var} and the weighted complex Bochner formula (13.9)
in \cite{Pal2}, we obtain
\begin{equation}
  \label{div-sec-var-met} \nabla_g^{\ast_{\Omega}}  \frac{d}{d t} _{\mid_{t =
  0}}  \dot{g}^{\ast}_t = \overline{\partial}^{\ast_{g, \Omega}}_{T_{X, J}}
  \overline{\partial}_{T_{X, J}} \nabla_{g, J}  \overline{\psi} = \frac{1}{2}
  \nabla_{g, J} \overline{ (\Delta^{\Omega}_{g, J} - 2\mathbbm{I}) \psi},
\end{equation}
and thus
\begin{eqnarray*}
  - \tmop{div}^{\Omega} \nabla_g^{\ast_{\Omega}} \frac{d}{d t} _{\mid_{t = 0}}
  \dot{g}^{\ast}_t  & = & \frac{1}{2} \Delta^{\Omega}_g R_{\psi} +
  \frac{1}{2} B^{\Omega}_{g, J} I_{\psi},\\
  &  & \\
  R_{\psi} & \assign & \tmop{Re} \left[ (\Delta^{\Omega}_{g, J} -
  2\mathbbm{I}) \psi \right],\\
  &  & \\
  I_{\psi} & \assign & \tmop{Im} \left[ (\Delta^{\Omega}_{g, J} -
  2\mathbbm{I}) \psi \right] .
\end{eqnarray*}
(Here we use the notation $z = \tmop{Re} z + i \tmop{Im} z$, for any $z \in
\mathbbm{C}$). Differentiating the tangential identity $2 d d_{J_t}^c 
\dot{\Omega}^{\ast}_t = - d \left[ \nabla_{g_t}^{\ast_{\Omega_t}}
\dot{g}_t^{\ast} \neg \omega \right]$, we obtain,
\begin{eqnarray*}
  2 d d_J^c  \frac{d}{d t} _{\mid_{t = 0}}  \dot{\Omega}^{\ast}_t  & = & - d
  \left[ \frac{d}{d t} _{\mid_{t = 0}} \left( \nabla_{g_t}^{\ast_{\Omega_t}}
  \dot{g}_t^{\ast} \right) \neg \omega \right] .
\end{eqnarray*}
Using the variation formula (\ref{super-var-Div}), and the identity
(\ref{div-sec-var-met}) we obtain
\begin{eqnarray*}
  \frac{d}{d t} _{\mid_{t = 0}} \left( \nabla_{g_t}^{\ast_{\Omega_t}}
  \dot{g}_t^{\ast} \right) & = & \frac{1}{4} \nabla_g | \dot{g}_0 |_g^2 +
  \nabla_g^{\ast_{\Omega}}  \frac{d}{d t} _{\mid_{t = 0}}  \dot{g}^{\ast}_t\\
  &  & \\
  & = & \nabla_g |A|_g^2 + \frac{1}{2} \nabla_{g, J} \overline{
  (\Delta^{\Omega}_{g, J} - 2\mathbbm{I}) \psi} .
\end{eqnarray*}
and thus
\begin{eqnarray*}
  \frac{d}{d t} _{\mid_{t = 0}}  \dot{\Omega}^{\ast}_t & = & - \frac{1}{2}
  R_{\psi} - |A|_g^2 + \int_X |A|_g^2 \Omega .
\end{eqnarray*}
We obtain in conclusion the variation formula
\begin{eqnarray*}
  2 \frac{d^2}{d t^2} _{\mid_{t = 0}}  \underline{H}_{g_t, \Omega_t} & = & 2
  \int_X \left| A \right|_g^2 F \Omega - 2 \left\langle \mathcal{L}^{\Omega}_g
  A, A \right\rangle_g\\
  &  & \\
  & - & (\Delta^{\Omega}_g - 2\mathbbm{I}) \left| A \right|_g^2 - 2 \int_X
  \left| A \right|_g^2 \Omega\\
  &  & \\
  & - & \frac{1}{2} (\Delta^{\Omega}_g - 2\mathbbm{I}) R_{\psi} + \frac{1}{2}
  B^{\Omega}_{g, J} I_{\psi}\\
  &  & \\
  & = & - 2 \left\langle J \nabla_g f \neg \nabla_g A, J A \right\rangle_g +
  4 \left\langle \nabla^2_g f, A^2 \right\rangle_g + 2 \int_X \left| A
  \right|_g^2 F \Omega\\
  &  & \\
  & - & (\Delta^{\Omega}_g - 2\mathbbm{I}) \left| A \right|_g^2 - 2 \int_X
  \left| A \right|_g^2 \Omega - \frac{1}{2} P^{\Omega}_{g, J} \tmop{Re} \psi,
\end{eqnarray*}
thanks to identity (\ref{stab-harm}) and a computation in the proof of lemma
25 in \cite{Pal2}. We denote respectively by $\pi_1$ and $\pi_2$ the projection to
the first and second factor of the decomposition (\ref{dec-P-op}). Then the
identity
\begin{eqnarray*}
  0 = \pi_2  \ddot{\gamma}_0 & = & \pi_2  \frac{d^2}{d t^2} _{\mid_{t = 0}} 
  \underline{H}_{g_t, \Omega_t},
\end{eqnarray*}
is equivalent to the identity
\begin{equation}
  \label{first-obstr}  \int_X u_1 \left[ 4 \left\langle \nabla^2_g f, A^2
  \right\rangle_g - 2 \left\langle J \nabla_g f \neg \nabla_g A, J A
  \right\rangle_g - (\Delta^{\Omega}_g - 2\mathbbm{I})_{_{_{_{_{}}}}} \left| A
  \right|_g^2 \right] \Omega = 0,
\end{equation}
for any $u = u_1 + i u_2 \in \Lambda^{\Omega}_{g, J}$, with $u_1$, $u_2$, real
valued. We observe now the equalities
\begin{eqnarray*}
  \int_X u_1  (\Delta^{\Omega}_g - 2\mathbbm{I})_{_{_{_{_{}}}}} \left| A
  \right|_g^2 \Omega & = & - \int_X B^{\Omega}_{g, J} u_2  \left| A
  \right|_g^2 \Omega\\
  &  & \\
  & = & \int_X u_2 B^{\Omega}_{g, J} \left| A \right|_g^2 \Omega\\
  &  & \\
  & = & \int_X u_2  \left( J \nabla_g f \right) . \left| A \right|_g^2
  \Omega\\
  &  & \\
  & = & 2 \int_X u_2  \left\langle J \nabla_g f \neg \nabla_g A, A
  \right\rangle_g \Omega .
\end{eqnarray*}
We conclude that the identity (\ref{first-obstr}) is equivalent to
\begin{eqnarray*}
  2 \int_X u_1  \left\langle \nabla^2_g f, A^2 \right\rangle_g \Omega & = &
  \int_X \left\langle J \nabla_g f \neg \nabla_g A, i \overline{u} \times_{_J}
  A_{_{_{_{}}}} \right\rangle_g \Omega,
\end{eqnarray*}
which shows the required conclusion.

\section{Appendix}
\subsection{Proof of the identities (\ref{div-Ev}) and (\ref{div-Tr})}

By
  definition of the $\Omega$-divergence operator and using the symmetry of $A$
  we infer
  \begin{eqnarray*}
    \tmop{div}^{\Omega} \left( A \cdot \xi \right) & = & g (\nabla_{g, e_k}
    \left( A \cdot \xi \right), e_k) - g \left( A \cdot \xi, \nabla_g f
    \right)\\
    &  & \\
    & = & g (\nabla_{g, e_k} A \cdot \xi + A \cdot \nabla_{g, e_k} \xi, e_k)
    - g \left( \xi, A \cdot \nabla_g f \right)\\
    &  & \\
    & = & g \left( \xi, \nabla_{g, e_k} A \cdot e_k - A \cdot \nabla_g f
    \right) \noplus + g \left( \nabla_{g, e_k} \xi, A e_k \right),
  \end{eqnarray*}
  and thus the identity (\ref{div-Ev}). We expand now the term
  \begin{eqnarray*}
    \tmop{div}^{\Omega} \tmop{Tr}_g \left( \nabla_g A \cdot A \right) & = &
    \tmop{div}^{\Omega}  \left( \nabla_{g, e_k} A \cdot A e_k \right)\\
    &  & \\
    & = & g \left( \nabla_{g, e_l} \left( \nabla_{g, e_k} A \cdot A e_k
    \right), e_l \right) - g \left( \nabla_{g, e_k} A \cdot A e_k, \nabla_g f
    \right)\\
    &  & \\
    & = & g \left( \nabla_{g, e_l} \nabla_{g, e_k} A \cdot A e_k + \nabla_{g,
    e_k} A \cdot \nabla_{g, e_l} A \cdot e_k, e_l \right)\\
    &  & \\
    & - & g \left( A e_k, \nabla_{g, e_k} A \cdot \nabla_g f \right).
 \end{eqnarray*}
 Expanding further we infer
  \begin{eqnarray*}
    \tmop{div}^{\Omega} \tmop{Tr}_g \left( \nabla_g A \cdot A \right) 
    & = & g \left( A e_k \nocomma, \nabla_{g, e_l} \nabla_{g, e_k} A \cdot
    e_l \right) + g \left( \nabla_{g, e_l} A \cdot e_k, \nabla_{g, e_k} A
    \cdot e_l \right)\\
    &  & \\
    & - & g \left( A e_k, \nabla_{g, e_k} A \cdot \nabla_g f \right)\\
    &  & \\
    & = & g \left( A e_k \nocomma, \nabla_{g, e_l} \widehat{\nabla_g A}
    \left( e_l, e_k \right) - \widehat{\nabla_g A} \left( \nabla_g f, e_k
    \right) \right)\\
    &  & \\
    & + & \left\langle \widehat{\nabla_g A}, \nabla_g A \right\rangle_g,
  \end{eqnarray*}
  and thus the identity (\ref{div-Tr}).
\subsection{The Maurer-Cartan equation in the K\"ahler case}
We observe that for any vector spaces $V$ and $E$, we can define a
contraction operation
\begin{eqnarray*}
  \neg \; : \; \left( \Lambda^p V^{\ast} \otimes V \right) \times \left(
  \Lambda^q V^{\ast} \otimes E \right) & \longrightarrow & \Lambda^{p + q - 1}
  V^{\ast} \otimes E\\
  &  & \\
  \left( \alpha, \beta \right) & \longmapsto & \alpha \; \neg \; \beta
 ,
\end{eqnarray*}
by the expression
\begin{eqnarray*}
  \left( \alpha \; \neg \; \beta \right) \left( \xi \right) & \assign &
  \sum_{|I| = \deg \alpha} \varepsilon_I \beta (\alpha (\xi_I),
  \xi_{\complement I})  .
\end{eqnarray*}
This map restricts to
\begin{eqnarray*}
  \neg \; : \; \mathcal{E}^{0, p} \left( T_X^{1, 0} \right) \times
  \mathcal{E}^{r, q} & \longrightarrow & \mathcal{E}^{r - 1, p + q}
  .
\end{eqnarray*}
We notice indeed the identity $\alpha \; \neg \; \beta \; = \;
\bar{\zeta}_I^{\ast} \wedge \left( \alpha_I \; \neg \; \beta \right)$, where
$\alpha = \alpha_I \otimes \bar{\zeta}_I^{\ast}$, with $\left( \zeta_k \right)_k \subset
  C^\infty ( U, T_{X, J}^{1, 0})$ a local frame. (We use from now on the Einstein convention for sums). 
  Obviously, the contraction operation $\neg $, generalises the one used in the previous sections.
\begin{lemma}{\bf(Expression of the exterior Lie product).}
  Let $\left( X, J, \omega \right)$ be a K\"ahler manifold and let
  $\alpha,
  \beta \in C^{\infty} (X, \Lambda^{0, \bullet}_J T^{\ast}_X
  \otimes_{\mathbbm{C}} T_{X, J}^{1, 0})$. Then holds the identity
  \begin{eqnarray*}
    {}[\alpha, \beta] & = & \alpha \neg \partial^{\omega}_{T^{1, 0}_{X, J}}
    \beta - \left( - 1 \right)^{\left| \alpha \right| \left| \beta \right|}
    \beta \neg \partial^{\omega}_{T^{1, 0}_{X, J}} \alpha .
  \end{eqnarray*}
\end{lemma}

\begin{proof}
  In the case $\left| \alpha \right| = \left| \beta \right| = 0$, the identity
  follows from an elementary computation in geodesic holomorphic coordinates.
  In order to show the general case, let $\left( \zeta_k \right)_k \subset
  \mathcal{O} ( U, T_{X, J}^{1, 0})$ be a local frame. We consider
  the local expressions $\alpha = \alpha_K \otimes \bar{\zeta}^{\ast}_K$,
  $\beta = \beta_L \otimes_{} \bar{\zeta}_L^{\ast}$. Then
  \begin{eqnarray*}
    {}[\alpha, \beta] & = & [\alpha_K, \beta_L] \otimes \left(
    \bar{\zeta}^{\ast}_K \wedge \bar{\zeta}^{\ast}_L \right) \\
    &  & \\
    & = & \left( \alpha_K \neg \partial^{\omega}_{T^{1, 0}_{X, J}} \beta_L -
    \beta_L \neg \partial^{\omega}_{T^{1, 0}_{X, J}} \alpha_K \right) \otimes
    \left( \bar{\zeta}^{\ast}_K \wedge \bar{\zeta}^{\ast}_L \right) .
  \end{eqnarray*}
  The identity $\overline{\partial}_{T^{1, 0}_{X, J}} \zeta_k = 0$ implies
  $\partial_{_J}  \bar{\zeta}^{\ast}_K = 0$. We infer
  \begin{eqnarray*}
    \partial^{\omega}_{T^{1, 0}_{X, J}} \alpha & = & \partial^{\omega}_{T^{1,
    0}_{X, J}} \alpha_K \wedge \bar{\zeta}^{\ast}_K,
  \end{eqnarray*}
  and a similar local expression for $\beta$. Thus using the identity 
  $$
  \alpha
  \; \neg \; \gamma \; = \; \bar{\zeta}_K^{\ast} \wedge \left( \alpha_K \;
  \neg \; \gamma \right),
  $$
with $\gamma$ arbitrary, we deduce
  \begin{eqnarray*}
    \alpha \neg \partial^{\omega}_{T^{1, 0}_{X, J}} \beta & = & \left(
    \alpha_K \neg \partial^{\omega}_{T^{1, 0}_{X, J}} \beta_L \right) \otimes
    \left( \bar{\zeta}^{\ast}_K \wedge \bar{\zeta}^{\ast}_L \right),\\
    &  & \\
    \beta \neg \partial^{\omega}_{T^{1, 0}_{X, J}} \alpha & = & \left( \beta_L
    \neg \partial^{\omega}_{T^{1, 0}_{X, J}} \alpha_K \right) \otimes \left(
    \bar{\zeta}^{\ast}_L \wedge \bar{\zeta}^{\ast}_K \right)\\
    &  & \\
    & = & \left( - 1 \right)^{\left| \alpha \right| \left| \beta \right|}
    \left( \beta_L \neg \partial^{\omega}_{T^{1, 0}_{X, J}} \alpha_K \right)
    \otimes \left( \bar{\zeta}^{\ast}_K \wedge \bar{\zeta}^{\ast}_L \right),
  \end{eqnarray*}
  and thus the required conclusion.
\end{proof}

We deduce that over a K\"ahler manifold the Maurer-Cartan equation 
\begin{equation*}
  \overline{\partial}_{T^{1, 0}_{X, J_0}} \theta +
  \frac{1}{2}  \left[ \theta, \theta \right] = 0,
\end{equation*}
writes as
\begin{equation}
  \label{super-complexMCARTAN}  \overline{\partial}_{T^{1, 0}_{X, J}} \theta +
  \theta \neg \partial^{\omega}_{T^{1, 0}_{X, J}} \theta = 0 .
\end{equation}
We show below that we can rewrite the
Maurer-Cartan equation in equivalent real terms as
\begin{equation}
  \label{super-realMCARTAN}  \overline{\partial}_{T_{X, J}} \mu + \mu \neg
  \nabla^{1, 0}_{g, J} \mu = 0,
\end{equation}
or in more explicit terms
\begin{eqnarray*}
  \left( \mathbbm{I}+ \mu \right) \neg J \nabla_g \mu & = & \left(
  \mathbbm{I}+ \mu \right) J \neg \nabla_g \mu .
\end{eqnarray*}
In order to show (\ref{super-realMCARTAN}) we expand, for any $u, v \in T_X$,
the term
\begin{eqnarray*}
  \left( \theta \neg \partial^{\omega}_{T^{1, 0}_{X, J}} \theta \right) \left(
  u, v \right) & = & \partial^{\omega}_{T^{1, 0}_{X, J}} \theta \left( \theta
  u, v \right) + \partial^{\omega}_{T^{1, 0}_{X, J}} \theta \left( u, \theta v
  \right)\\
  &  & \\
  & = & \nabla^{1, 0}_{g, J} \theta \left( \theta u, v \right) - \nabla^{1,
  0}_{g, J} \theta \left( v, \theta u \right)\\
  &  & \\
  & + & \nabla^{1, 0}_{g, J} \theta \left( u, \theta v \right) - \nabla^{1,
  0}_{g, J} \theta \left( \theta v, u \right) .
\end{eqnarray*}
Expanding further we obtain
\begin{eqnarray*}
  2 \left( \theta \neg \partial^{\omega}_{T^{1, 0}_{X, J}} \theta \right)
  \left( u, v \right) & = & \nabla_g \theta \left( \theta u, v \right) - i
  \nabla_g \theta \left( J \theta u, v \right)\\
  &  & \\
  & - & \nabla_g \theta \left( v, \theta u \right) + i \nabla_g \theta \left(
  J v, \theta u \right)\\
  &  & \\
  & + & \nabla_g \theta \left( u, \theta v \right) - i \nabla_g \theta \left(
  J u, \theta v \right)\\
  &  & \\
  & - & \nabla_g \theta \left( \theta v, u \right) + i \nabla_g \theta \left(
  J \theta v, u \right) .
\end{eqnarray*}
Using the fact that {\tmem{$\theta$}} takes values in $T^{1, 0}_{X, J}$ we
obtain
\begin{eqnarray*}
  2 \left( \theta \neg \partial^{\omega}_{T^{1, 0}_{X, J}} \theta \right)
  \left( u, v \right) & = & 2 \nabla_g \theta \left( \theta u, v \right) -
  \nabla_g \theta \left( v, \theta u \right) + i \nabla_g \theta \left( J v,
  \theta u \right)\\
  &  & \\
  & - & 2 \nabla_g \theta \left( \theta v, u \right) + \nabla_g \theta \left(
  u, \theta v \right) - i \nabla_g \theta \left( J u, \theta v \right) .
\end{eqnarray*}
Replacing on the right hand side of this equality the identity $2 \theta = \mu
- i J \mu$ and adding the conjuguate of both sides we infer
\begin{eqnarray*}
 & & 8 \left( \theta \neg \partial^{\omega}_{T^{1, 0}_{X, J}} \theta \right)
  \left( u, v \right) + 8 \overline{\left( \theta \neg
  \partial^{\omega}_{T^{1, 0}_{X, J}} \theta \right) \left( u, v \right)} \\
  &  & \\
  & =
  & 4 \nabla_g \mu \left( \mu u, v \right) - 4 J \nabla_g \mu \left( J \mu u,
  v \right)\\
  &  & \\
  & - & 2 \nabla_g \mu \left( v, \mu u \right) + 2 J \nabla_g \mu \left( v, J
  \mu u \right)\\
  &  & \\
  & + & 2 \nabla_g \mu \left( J v, J \mu u \right) + 2 J \nabla_g \mu \left(
  J v, \mu u \right)\\
  &  & \\
  & + & 2 \nabla_g \mu \left( u, \mu v \right) - 2 J \nabla_g \mu \left( u, J
  \mu u \right)\\
  &  & \\
  & - & 2 \nabla_g \mu \left( J u, J \mu v \right) - 2 J \nabla_g \mu \left(
  J u, \mu v \right)\\
  &  & \\
  & - & 4 \nabla_g \mu \left( \mu v, u \right) + 4 J \nabla_g \mu \left( J
  \mu v, u \right) .
\end{eqnarray*}
Using the anti $J$-linearity of $\nabla_{g, \xi} \mu$ we deduce
\begin{eqnarray*}
& &  8 \left( \theta \neg \partial^{\omega}_{T^{1, 0}_{X, J}} \theta \right)
  \left( u, v \right) + 8 \overline{\left( \theta \neg
  \partial^{\omega}_{T^{1, 0}_{X, J}} \theta \right) \left( u, v \right)} \\
  &  & \\
  & = & 4 \nabla_g \mu \left( \mu u, v \right) - 4 J \nabla_g \mu \left( J \mu u,
  v \right)\\
  &  & \\
  & - & 4 \nabla_g \mu \left( \mu v, u \right) + 4 J \nabla_g \mu \left( J
  \mu v, u \right) \\
  &  & \\
  & = & 8 \nabla^{1, 0}_{g, J} \mu \left( \mu u, v \right) - 8 \nabla^{1,
  0}_{g, J} \mu \left( \mu v, u \right)\\
  &  & \\
  & = & 8 \left( \mu \neg \nabla^{1, 0}_{g, J} \mu \right) \left( u, v
  \right) .
\end{eqnarray*}
The latter combined with
\begin{eqnarray*}
  \overline{\partial}_{T^{1, 0}_{X, J}} \theta \left( u, v \right) +
  \overline{\overline{\partial}_{T^{1, 0}_{X, J}} \theta \left( u, v \right)}
  & = & \overline{\partial}_{T_{X, J}} \mu \left( u, v \right),
\end{eqnarray*}
and (\ref{super-complexMCARTAN}) implies the required identity
(\ref{super-realMCARTAN}).

\vspace{1cm}
\noindent
Nefton Pali
\\
Universit\'{e} Paris Sud, D\'epartement de Math\'ematiques 
\\
B\^{a}timent 425 F91405 Orsay, France
\\
E-mail: \textit{nefton.pali@math.u-psud.fr}

\end{document}